\documentclass[reqno,12pt]{amsart}
\usepackage[utf8]{inputenc}
\usepackage{amsmath,amssymb,amsthm,mathrsfs,color,times,textcomp,verbatim,yfonts}

\usepackage[left=1.1in, right=1.1in, top=1in]{geometry}

\usepackage[normalem]{ulem}
\usepackage{subcaption}
\usepackage[export]{adjustbox}
\usepackage{esint}
\usepackage{xcolor}
\usepackage{array}
\usepackage[colorlinks=true]{hyperref}

\hypersetup{urlcolor=blue, citecolor=red, linkcolor=blue}
\usepackage{mathtools}
\usepackage[shortlabels]{enumitem}
\usepackage[square,numbers]{natbib}

\numberwithin{equation}{section}
\theoremstyle{plain} 
\newtheorem{thm}{Theorem}[section]
\newtheorem{prop}[thm]{Proposition}
\newtheorem{lem}[thm]{Lemma}

\newtheorem{claim}{Claim}

\usepackage{graphicx}

\theoremstyle{definition}

\newtheorem{rem}[thm]{Remark}

\def\1{\mathbf{1}}
\def\for{\quad\text{ for}}
\def\Mod{\text{Mod}}

\title[Slow blow up for 4D heat equation]{A SLOW BLOW UP SOLUTION FOR THE FOUR DIMENSIONAL energy critical SEMI LINEAR HEAT EQUATION}

\author{Tongtong Li}
\address{Academy of Mathematics and Systems Science, the Chinese Academy of Sciences, Beijing 100190, China.}
\email{lttlydy@139.com}

\author{Liming Sun}

\email{lmsun@amss.ac.cn}

\author{Shumao Wang}
\email{wangshumao@amss.ac.cn}

\date{\today\, (Last Typeset)}

\begin{document}

\maketitle
\begin{abstract}
We consider the energy critical four dimensional semi-linear heat
equation 
\[
\partial_{t}v-\Delta v-v^{3}=0, \quad(t,x)\in \mathbb{R}\times \mathbb{R}^4.
\]

Formal computation of Filippas et al. (R. Soc. Lond. Proc. 2000) conjectures the existence of a sequence of type II blow-up solutions with various blow-up rates
\[
\|v(t)\|_{L^\infty(\mathbb{R}^4)}\approx \frac{|\log(T-t)|^{\frac{2L}{2L-1}}}{(T-t)^L}
,\quad L=1,2,\cdots.\]
Schweyer (J. Funct. Anal. 2012) rigorously constructs a type II blow-up solution for the case $L=1$. In this paper, we  show  the existence of type II
blow-up solution for $L=2$. The method here  could be generalized to deal with all the cases $L\geq 2$. 
\end{abstract}

\setlength{\parskip}{0.5em}

\section{Introduction}

\subsection{Setting of the problem}

Consider the following semi-linear heat equation 
\begin{equation}\label{flow-main}
\begin{cases}
\partial_{t}v-\Delta v=|v|^{p-1}v, & (t,x)\in\mathbb{R}\times\mathbb{R}^{d},\\
v\mid_{t=t_{0}}=v_{0}.
\end{cases}
\end{equation}
Because of the simplicity of the nonlinearity, problem \eqref{flow-main} has been widely considered as a popular model for testing the methods designed to analysis the behavior of solutions near singularity formation.
It has been extensively studied in the literature, for example \cite{giga1985asymptotically,giga1987characterizing,herrero1993blow,mizoguchi2004type,collot2017dynamics}. It is well-known that for a large class of initial data (for instance, bounded and continuous) there is a unique maximal classical solution  $v(t,x)$ for $t\in (0,T)$. If $T$ is finite, then $u$ will blow up at time $t=T$. There are two types of blow-ups depending on the rate
\[
\begin{cases}
\limsup_{t\rightarrow T}(T-t)^{\frac{1}{p-1}}||v(t)||_{L^{\infty}(\mathbb{R}^{d})}<+\infty & \text{type I},\\
\limsup_{t\rightarrow T}(T-t)^{\frac{1}{p-1}}||v(t)||_{L^{\infty}(\mathbb{R}^{d})}=+\infty & \text{type II}.
\end{cases}
\]

In this article, we will focus on the radial type II blow-up solution in energy critical case, that is $p=p_{S}:=\frac{d+2}{d-2}$. In this
situation, the total dissipated energy 
\begin{equation}
E(v)=\frac{1}{2}\int_{\mathbb{R}^{d}}|\nabla v|^{2}dx-\frac{1}{p+1}\int_{\mathbb{R}^{d}}|v|^{p+1}dx
\end{equation}
is left invariant by the scaling symmetry of the problem 
\[
v(t,x)\mapsto\lambda^{\frac{2}{p-1}}v\left(\lambda^{2}t,\lambda x\right).
\]

The blow-up of \eqref{flow-main} is almost completely understood in the sub-critical range $1<p<p_S$, for instance, by \cite{filippas1992refined, giga1985asymptotically,giga1987characterizing,giga2004blow,quittner1999priori,velazquez1992higher}. The solution always blows up in type I in this range. The existence of type II blow-up has been established in various settings, for instance by  \cite{herrero1993blow,herrero1994explosion,mizoguchi2004type} when $p>p_{JL}$, where 
\begin{align}
    p_{{JL}}=\begin{cases}
\infty & \text { if } d \leq 10, \\
1+\frac{4}{d-4-2 \sqrt{d-1}} & \text { if } d \geq 11 .
\end{cases}
\end{align}
On the other hand, when $p_S<p<p_{JL}$,  \citet{matano2004nonexistence} excludes the occurrence of a type II blow-up for radial solutions. Therefore, in dimension $3\leq d\leq 10$, the choice $p = p_S$ is the only one for which a type II blow-up occurs for radial data.

Recently, there are active researches in the energy critical case $p=p_S$. In the pioneering work by \citet*{filippas2000fast}, they find that $u$ can exhibit type II blow-up in finite time in lower dimensions in the energy critical case $p=p_S$.  They formally obtain sign changing type II blow-up solutions using the matched asymptotic expansion technique.
Also, they give a sequence of blow-up speeds (corrected by \citet{harada2020type}) 
\begin{align}\label{blow-rate-all}
||v(t)||_{L^{\infty}(\mathbb{R}^{d})}\approx\begin{cases}
(T-t)^{-L} & d=3,\\
\frac{|\log(T-t)|^{\frac{2L}{L-1}}}{(T-t)^{L}} & d=4,\\
(T-t)^{-3L} & d=5,\\
{(T-t)^{-\frac{5}{2}}|\log(T-t)|^{-\frac{15}{4}}} & d=6,
\end{cases}
\end{align}
where $L\geq 1$ is an integer. Recently, there is a surge of interest in constructing such type II blow-up solutions as predicted by \cite{filippas2000fast}.

\citet{schweyer2012type} first rigorously constructs a radial blow-up 
solution in the case $d=4$ and $L=1$. He uses a strategy developed
in the study of geometrical dispersive problems by \citet{merle2005blow-up,merle2006sharp,merle2011} and \citet{raphael2012stable}.  The nature of his approach is energy estimates and making no use of the maximum principle. Meanwhile, del Pino, Wei and their collaborators develop an inner-outer gluing method  and lead to a series of works on construction. They apply their methods to construct type II blow-up solutions in several cases. To be more precise, \citet{del2020type} constructs solutions of $d=3$ and all $L\geq 1$.   \citet{del2019type} establishes the existence of solutions of $d=5$ and $L=1$. Later on,  \citet{harada2020higher} completes the construction for $d=5$ and all $L\geq 2$. Using this inter-outer gluing method, \citet{harada2020type} also shows the existence of type II blow-up solution with the specific rate for $d=6$ in \eqref{blow-rate-all}.


One may wonder what will happen for $d\geq 7$.  \citet*{collot2017dynamics} 
proves no existence of type II blow-up solution in $d\geq7$ cases near the ground state
solitary wave
and gives a complete classification of its asymptotic behavior. Recently, \citet{wang2021refined} precludes the type II blow-up for all \textit{positive} solutions in  $d\geq 7$.  

After all the works mentioned above, it seems that the cases $d=4$ and $L\geq 2$ are still unsettled. The goal of this paper is to fill the gap of these remaining cases.


\subsection{Statement of the result}

We consider the energy critical semi-linear heat equation in dimension $d=4$
\begin{equation}
\begin{cases}
\partial_{t}v-\Delta v-v^{3}=0, & (t,x)\in\mathbb{R}\times\mathbb{R}^{4},\\
v\mid_{t=t_{0}}=v_{0}.
\end{cases}\label{eq1}
\end{equation}
In this article, we construct a radial type II blow-up solution based
on the Talenti-Aubin soliton 
\begin{equation}
Q(r)=\frac{1}{1+\frac{r^{2}}{8}},\quad \Delta Q+Q^{3}=0.
\end{equation}

Our result is the following 
\begin{thm}
\label{thm:main}For any $\alpha^{*}>0$, there exists $C^{\infty}$
radial initial data $v_{0}$ with 
\begin{equation}
E(Q)<E(v_{0})<E(Q)+\alpha^{*}
\end{equation}
such that the solution to (\ref{eq1}) blows up in finite time $T=T(v_{0})<\infty$
in a type II regime: there exists $v^{*}\in\dot{H}^{1}$ such that 
\begin{equation}
v(t,x)-\frac{1}{\lambda(t)}Q\left(\frac{x}{\lambda(t)}\right)\rightarrow v^{*}\quad \text{in}\ \dot{H}^{1}\ \text{as}\ t\to T,
\label{1.8}\end{equation}
and for some $c(v_0)>0$, 
\begin{equation}
\lambda(t)=c(v_{0})(1+o_{t\to T}(1))\frac{(T-t)^{2}}{|\log(T-t)|^{\frac{4}{3}}}.
\label{1.9}\end{equation}
\end{thm}

\begin{rem}
(i) Our method relies on the approach to construct slow blow-up dynamics
for the corotational energy-critical harmonic heat flow of \citet{raphael2015quantized}. People have observed a deep connection between dimension four
energy critical semi-linear heat equation and two-dimensional harmonic map flows, and interested
readers can check \cite{raphael2015quantized,davila2020singularity} and references therein. Our result verifies the existence of the blow-up speed corresponding to $L=2$ as conjectured
in \cite{filippas2000fast}
because $v$ will blow up in the speed of the reciprocal of $\lambda$.

(ii) The construction to the case of $L=2$ is much more complicate than the one of $L=1$. First, the approximate solution in the case of $L=1$ needs to be sharpened here in order to get \eqref{1.9}. This requires a better approximation to the blow-up solution to further reduce the errors which are produced in the case of $L=1$. See Step 1 in subsection \ref{subsection:strategy of proof} for more details.  Second, we need to deal with two different ``unstable directions''
in our setting. One is from the Schr\"{o}dinger operator $H$ in (\ref{def:H}), and the
other one is from modulation parameter $b=(b_{1},b_{2})$ in (\ref{roughModeq-1}).
In this problem, we can deal with them at the same time. See Step 3 in subsection \ref{subsection:strategy of proof}. 
For the case of $L> 2$, these two types of difficulty persist and are all the essential ones. Actually, one can introduce an appropriate class of functions to continue the approximate process (see a similar argument in \cite{raphael2015quantized}). This will also produce more unstable directions but can be handled similarly as here. Since the proof of $L>2$ is just a tautology of the idea here, we will not present it in this paper.

(iii) The method we rely on is a powerful tool. It has been applied to construct finite time blow-up solutions in Schr\"{o}dinger map \cite{merle2013blowup},  focusing energy supercritical Schr\"{o}dinger equation \cite{merle2015type}, defocusing energy supercritical Schr\"{o}dinger equation in \cite{merle2022blow},  energy supercritical wave equation in \cite{collot2018type}, nonradial energy supercritical heat equation in \cite{collot2017nonradial}.


\end{rem}

\subsection{Notations }

We introduce the differential operator 
\[
\Lambda f:=f+y\cdotp\nabla f\quad \text{(energy critical scaling)}.
\]
Given a positive number $b_{1}>0$, we let 
\[
B_{0}:=\frac{1}{\sqrt{b_{1}}},\ B_{1}:=\frac{|\log b_{1}|}{\sqrt{b_{1}}}.
\]
Given a parameter $\lambda>0$, we let 
\[
u_{\lambda}(r):=\frac{1}{\lambda}u(y)\quad \text{with}\ y=\frac{r}{\lambda}.
\]
We let $\chi$ be a smooth non-increasing cutoff function with 
\[
\chi(y)=\begin{cases}
1 & \text{for}\ y\leq1,\\
0 & \text{for}\ y\geq2,
\end{cases}
\]
and use the notation 
\[
\int f:=\int_{0}^{+\infty}f(r)r^{3}dr.
\]

\subsection{Strategy of the proof}\label{subsection:strategy of proof}

Now we sketch the main points in the proof of Theorem \ref{thm:main}.

\emph{\uline{Step1: Approximate solution $Q_{b}$ .}} We first reparametrize (\ref{eq1})
by
\begin{equation}
v(t,r)=\frac{1}{\lambda}u(s,y),\quad y=\frac{r}{\lambda(t)},\quad \frac{ds}{dt}=\frac{1}{\lambda^{2}(t)}
\end{equation}
which leads to
\begin{equation}
\partial_{s}u+b_{1}\Lambda u-\Delta u-u^{3}=0,\ b_{1}=-\frac{\lambda_{s}}{\lambda}.
\end{equation}
We look for blow-up solution $u$ close to $Q$ in $\dot{H}^{1}$
topology. In this case, $b_{1}$ remains small and the flow is controlled by the linearized Hamiltonian
\begin{equation}
H=-\Delta-V=-\Delta-3Q^{2}
\end{equation}
which has a resonance by the scaling symmetry
\begin{equation}
H\Lambda Q=0,\quad \Lambda Q\sim-\frac{8}{y^{2}}\quad \text{as}\quad y\to+\infty.
\end{equation}
Besides, we need $T_{1},T_{2}$ given by
\begin{equation}
HT_{2}=-T_{1},\ HT_{1}=-\Lambda Q
\end{equation}
with asymptotics
\begin{equation}
T_{1}\sim-4\log y+2,\ T_{2}\sim y^{2}\left(\frac{1}{2}\log y -\frac{5}{8}\right)
\end{equation}
as $y$ tends to infinity.

The linearization of the flow implies a possible approximate solution
\begin{equation}
Q+b_{1}T_{1}+b_{2}T_{2}
\end{equation}
with \textit{a priori} bound $|b_{2}|\lesssim b_{1}^{2}$.
Indeed, the $O(b_{1})$ error vanishes 
\begin{equation}
b_{1}(\Lambda Q+HT_{1})=0
\end{equation}
by the definition of $T_{1}$.
At $O(b_{1}^{2})$ level, the leading term is
\begin{equation}
(b_{1})_{s}T_{1}+b_{1}^{2}\Lambda T_{1}-b_{2}T_{1},
\end{equation}
which can, since $\Lambda T_{1}\sim T_{1}$, be canceled by setting
\begin{equation}
(b_{1})_{s}+b_{1}^{2}-b_{2}=0.
\end{equation}
Similar calculation at $O(b_{1}^{3})$ level suggests that we take
\begin{equation}
(b_{2})_{s}+3b_{1}b_{2}=0.
\end{equation}
Certainly some other terms $S_{j}(b,y)$ are needed to further reduce the
error, and we construct the approximate solution as
\begin{equation}\label{Q-b}
Q_{b}(y)=Q(y)+b_{1}T_{1}+b_{2}T_{2}+S_{2}+S_{3}+S_{4}=Q(y)+\alpha(b,y).
\end{equation}
which generates a small error $\Psi_{b}$. More importantly, from
flux computation, some $\log b_{1}$ term should be introduced to the
dynamical system for $b=(b_{1},b_{2})$. That is, we should take
\begin{equation}
\begin{cases}
\frac{\lambda_{s}}{\lambda}+b_{1}=0,\\
(b_{1})_{s}+b_{1}^{2}(1+c_{b_1})-b_{2}=0, & \text{with}\quad c_{b_1}\sim\frac{2}{|\log b_{1}|}\\
(b_{2})_{s}+b_{1}b_{2}(3+c_{b_1})=0.
\end{cases}\label{roughModeq-1}
\end{equation}
Indeed, this ODE system has a solution with $\lambda\to0^{+}$ at
finite time $T=T(v_{0})<+\infty$ and 
\begin{equation}
\lambda\sim c_{1}(v_{0})\frac{(\log s)^{\frac{4}{9}}}{s^{\frac{2}{3}}},\  b_{1}\sim b_{1}^{e}=\frac{2}{3s}-\frac{4}{9s\log s},\ 
b_{2}\sim b_{2}^{e}=-\frac{2}{9s^{2}}+\frac{20}{27s^{2}\log s}.
\end{equation}
 The problem is that this solution might be unstable. In fact, we define $U$ by
\begin{equation}\label{decomb-1}
b_{k}=b_{k}^{e}+\frac{U_{k}}{s^{k}(\log s)^{\frac{5}{4}}},\quad U=\begin{bmatrix}U_1\\
U_2
\end{bmatrix}.
\end{equation}
Then the dynamical system for $b$ above implies
\begin{equation}
    s\frac{dU}{ds}=AU+\text{error terms,}
\end{equation}
which is, after diagonalization, equivalent to
\begin{equation}
    s\frac{dV}{ds}=D_{A}V+\text{error terms},\quad D_{A}=\begin{bmatrix}
    -1& \\
      & \frac{2}{3}
    \end{bmatrix},\ V=PU
\end{equation}
where the first unstable direction $V_2$ occurs corresponding to the positive eigenvalue $\frac{2}{3}$ of $A$. And it must be controlled \emph{a priori} to avoid disrupting the dynamic for $b$. For technical reasons we modify it to $\tilde{V}_{2}$, see Proposition \ref{prop:improved modulation} and Lemma \ref{lemma：dynamicV'}.

\emph{\uline{Step 2: Decomposition of the flow.}} To get (\ref{roughModeq-1}), we decompose $u$ as
\begin{equation}
u=Q_{b}+\varepsilon
\end{equation}
for some small $\varepsilon(s,y)$ subject to the orthogonal
conditions
\begin{equation}
(\varepsilon,H^{k}\Phi_{M})=0\ \for\ 0\le k\le2.\label{orthoepsilon-1}
\end{equation}
where $(\ ,\ )$ denotes $L^2$ inner product and $\Phi_{M}$ is a compactly supported substitute for $\Lambda Q$ (since $\Lambda Q\notin L^2$).
Then Implicit Function Theorem ensures the existence and uniqueness
of the decomposition as long as $\|\varepsilon\|_{\dot{H}^1}$ remains small.
Moreover, it turns out that (\ref{orthoepsilon-1}) is enough to derive
the expected modulation equations
\begin{equation}
\left|\frac{\lambda_{s}}{\lambda}+b_{1}\right|+
|(b_{1})_{s}+b_{1}^{2}(1+c_{b_1})-b_{2}|+
|(b_{2})_{s}+b_{1}b_{2}(3+c_{b_1})|
\lesssim\|\varepsilon\|_{loc}+\frac{b_{1}^{3}}{|\log b_{1}|}
\end{equation}
for some local-in-space norm $\|\varepsilon\|_{loc}$. 

So we need to bound $\|\varepsilon\|_{loc}$. Thanks
to the Hardy type bounds below ensured by (\ref{orthoepsilon-1})
\begin{equation}
\varXi_{2k}:=\int|H^{k}\varepsilon|^{2}\gtrsim\int\frac{|\varepsilon|^{2}}{(1+y^{4k})(1+|\log y|^{2})},\quad 1\le k\le 3,
\end{equation}
we turn to bound $\varXi_{2k}$. Here comes the second unstable direction, i.e. the projection 
\begin{equation}\label{deftau-1}
\tau(t)=(\varepsilon,\psi).
\end{equation}
Here $\psi$ corresponds to the only non-positive spectrum of $H$
\begin{equation}
H\psi=-\varsigma\psi,\ \varsigma>0.
\end{equation}
If uncontrolled, its growth would destroy any bound of $\varXi_{2k}$.
Assuming suitable \emph{a priori} bound on $\tau$, we can derive some Lyapounov monotonicity
\begin{equation}
\frac{d}{dt}\frac{\varXi_{6}}{\lambda^{10}}
\lesssim\frac{1}{\lambda^{12}}\frac{b_{1}^{6}}{|\log b_{1}|^{2}}
\end{equation}
which leads to the estimate
\begin{equation}
\varXi_{6}\lesssim\frac{b_{1}^{6}}{|\log b_{1}|^{2}}.
\end{equation}
This together with similar bounds on $\varXi_{4}$ and $\varXi_{2}$
serves to control $\|\varepsilon\|_{loc}$. The required smallness of $\|\varepsilon\|_{\dot{H}^{1}}$
is ensured by the dissipation of energy and the sub-coercivity of $H$.

\emph{\uline{Step 3: Control of the two unstable directions.}} With the above analysis, we see the core of the proof is to control
the unstable models $\tilde{V}_{2}$ and $\tau$ \emph{at the same time}. This
is possible since $\tilde{V}_{2}$ comes from the development of modulation parameter $b$ while $\tau$ comes from
the projection of $\varepsilon$ to the non-positive eigenvalue direction of $H$. 
Hence they are (almost) independent of each other as long as both are reasonably small so as not to destroy the bounds for $b$ and $\varepsilon$.

The method is to apply a Brouwer type argument. We set initial data $v_0$ as
\begin{equation}
    v_{0}=Q_{b(0)}+\varepsilon(0)=Q_{b(0)}+\tau(0) \tilde{\psi}
\end{equation}
where $\tilde{\psi}$ satisfies
\begin{equation}
(\tilde{\psi},\psi)=1,\quad (\tilde{\psi},H^{k}\Phi_{M})=0,\quad 0\le k\le2.
\end{equation}
so that (\ref{orthoepsilon-1}) and (\ref{deftau-1}) are satisfied at $t=0$. It can be seen that $v_0$ is uniquely determined by $U_{1}(0)$, $\tilde{V}_{2}(0)$ and $\tau(0)$. We simply take $
U_{1}(0)=0$ and assume \emph{a priori} 
\begin{equation}\label{unstable}
\tilde{V}_{2}(t)\in [-1,1],\quad
\tilde{\tau}(t):=\tau(t)\cdotp\frac{\log b_{1}(t)}{b_{1}(t)^{3+\frac{1}{2}}}\in [-1,1].
\end{equation}
If (\ref{unstable}) fails at $t=\tilde{T}_{exit}(v_{0})<T(v_{0})$
for any reasonable initial data $v_{0}$, we will have a map
\begin{align*}
\mathbb{D}=[-1,1]\times[-1,1] & \to\partial\mathbb{D},\\
(\tilde{V}_{2}(0),\tilde{\tau}(0)) & \mapsto(\tilde{V}_{2}(\tilde{T}_{exit}),\tilde{\tau}(\tilde{T}_{exit})).
\end{align*}
Moreover, the dynamics for $b$ and $\varepsilon$ yield, respectively,
\begin{equation}
|s(\tilde{V}_{2})_{s}-\frac{2}{3}\tilde{V}_{2}|\lesssim{(\log s)^{-\frac{1}{4}}},\ |\tau_{s}-\varsigma\tau|\lesssim\frac{b_{1}^{4}}{|\log b_{1}|},
\end{equation}
which ensure the strictly outgoing behavior 
\begin{equation}
\frac{d}{ds}\tilde{V}_{2}^{2}(\tilde{T}_{exit})>0,\quad  \frac{d}{ds}\tilde{\tau}^{2}(\tilde{T}_{exit})>0.
\end{equation}
Hence classical PDE theory ensures that the map above is continuous
and leaves the boundary points fixed, which contradicts Brouwer fixed-point theorem. So we
conclude that some $v_{0}$ exists such that (\ref{unstable}) holds
for all $t<T(v_{0})$, hence deduce (\ref{roughModeq-1}) and Theorem \ref{thm:main} .

The article is organized as follows. In Section \ref{APPROXIMATE PROFILE}, we first construct the approximate self-similar solution $Q_{b}$, give the sharp estimates about the error term $\Psi_{b}$ and supply a local version of $Q_{b}$. Then we derive the dynamical system of $b=(b_{1},b_{2})$ in Section \ref{Sec:Dynamics for b}, which corresponds to one of the unstable directions. In Section \ref{THE TRAPPED REGIME}, we first give a suitable decomposition of the solution $v(t,x)$ and design the bootstrap regime. Then we derive the modulation equations in Section \ref{Modulation equations}. Finally, in the end of this section, we derive the fundamental monotonicity of the Sobolev-type norms $\varXi_{2},\varXi_{4}$ and $\varXi_{6}$. In Section \ref{CLOSING THE BOOTSTRAP AND PROOF OF THEOREM 1.1}, we first get improved control of $\varXi_{2},\varXi_{4}$ and $\varXi_{6}$. Then by a standard Brouwer argument, we control the two unstable directions at the same time by suitably choosing the initial data, which is the heart of our analysis. The above analysis finishes the bootstrap regime, which easily implies the blow up statement of Theorem \ref{thm:main}.

\section{APPROXIMATE PROFILE}\label{APPROXIMATE PROFILE}

Introduce a parameter $\lambda=\lambda(t)>0$. Let $v(t,\cdotp)=u_{\lambda}(s,\cdotp)$,
or 
\begin{equation}
v(t,r)=\frac{1}{\lambda}u(s,y),\quad y=\frac{r}{\lambda}
\end{equation}
where 
\begin{equation}
s=s_{0}+\int_{0}^{t}\frac{d\tau}{\lambda^{2}(\tau)},\quad s_{0}>0.
\end{equation}
then (\ref{eq1}) becomes 
\begin{equation}
\partial_{s}u-\frac{\lambda_{s}}{\lambda}\Lambda u-\Delta u-u^{3}=0.\label{eq2}
\end{equation}

In this section we construct an explicit approximate solution to (\ref{eq2})
close to $Q$. The construction relies on the spectral properties
of linearized Hamiltonian 
\begin{equation}\label{def:H}
H=-\Delta-V=-\Delta-3Q^{2}
\end{equation}
which are well known and are summarized below:

(i) $H$ has a unique negative eigenvalue 
\begin{equation}
H\psi=-\varsigma\psi,\quad \varsigma>0\label{def:psi}
\end{equation}
and $\psi$ decays exponentially.

(ii) $H$ has a resonance at the origin induced by the scaling symmetry
\begin{equation}
H\Lambda Q=0
\end{equation}
with $\Lambda Q\notin L^{2}$, or more precisely 
\begin{equation}\label{Q}
\for\ 0\le i\le 10,\quad \Lambda^{i}Q=\begin{cases}
1+O(y^{2}) & \text{as}\ y\to0,\\
(-1)^{i}\frac{8}{y^{2}}+O(\frac{1}{y^{4}}) & \text{as}\ y\to+\infty.
\end{cases}
\end{equation}
ODE theory provides another solution to $H\Gamma=0$ for $y>0$:
\begin{equation}
\begin{split}
\Gamma(y)& =-\Lambda Q(y)\int_{1}^{y}\frac{dx}{x^{3}[\Lambda Q(x)]^{2}}\\
 &=\frac{y^{2}-8}{\left(y^{2}+8\right)^{2}}
 \left(\frac{y^{2}}{16}+6 \log y-\frac{583}{112}-\frac{4}{y^{2}}\right)
 -\frac{64}{\left(y^{2}+8\right)^{2}}
 \end{split}
\end{equation}
and the asymptotic behavior 
{
\begin{equation}
\for\ 0\le i\le 10,\quad \Lambda^{i}\Gamma(y)=\begin{cases}
(-1)^i\frac{2}{y^2}+O(|\log y|) & \text{as}\ y\to0,\\
\frac{1}{16}+O(\frac{\log y}{y^{2}}) & \text{as}\ y\to+\infty.
\end{cases}\label{Gamma}
\end{equation}}

(iii) By (ii), the solution to $Hu=f$ is given by 
\begin{equation}\label{Hu=f:sol}
u=H^{-1}f=\Gamma(y)\int_{0}^{y}f(x)\Lambda Q(x)x^3dx-\Lambda Q(y)\int_{0}^{y}f(x)\Gamma(x)x^3dx
\end{equation}
up to the addition of $\alpha Q(y)+\beta\Gamma(y)$. We restrict ourselves to the case where $f$ is smooth and take $\alpha=\beta=0$,
which is equivalent to require $u$ is smooth and $u(0)=0$.

Now we turn to the construction. In the following subsection, we assume
\begin{equation}
\begin{cases}
\frac{\lambda_{s}}{\lambda}+b_{1}=0,\\
(b_{1})_{s}+b_{1}^{2}(1+c_{b_1})-b_{2}=0,\\
(b_{2})_{s}+b_{1}b_{2}(3+c_{b_1})=0.
\end{cases}\label{roughModeq}
\end{equation}
where $c_{b_1}$ is defined by (\ref{c_b}).

\subsection{Construction of the approximate blow-up profile}
\begin{prop}
\label{prop:approximate}(Construction of the approximate profile).
Let $M>0$ be large enough. Then there exists a small enough universal
constant $b^{*}(M)>0$, such that the following holds true. Let there
be a $C^{1}$ map 
\[
b=(b_{k})_{1\leq k\leq2}:[s_{0},s_{1}]\longmapsto(-b^{*}(M),b^{*}(M))\times (-b^{*}(M),b^{*}(M))
\]
with \textit{a priori} bounds on $[s_{0},s_{1}]$ 
\begin{equation}
0<b_{1}<b^{*}(M),\quad |b_{2}|\lesssim b_{1}^{2}.\label{apriori1}
\end{equation}
Then there exist profiles $T_{1},T_{2},S_{2},S_{3}$ and $S_{4}$,
such that 
\begin{equation}
Q_{b}(y)=Q(y)+b_{1}T_{1}+b_{2}T_{2}+S_{2}+S_{3}+S_{4}=Q(y)+\alpha(b,y)
\end{equation}
generates an error 
\begin{equation}
\Psi_{b}:=\partial_{s}Q_{b}-\frac{\lambda_{s}}{\lambda}\Lambda Q_{b}-\Delta Q_{b}-Q_{b}^{3}
\end{equation}
which satisfies:
\begin{align}
\int_{y\leq2B_{1}}|H^{k}\Psi_{b}|^{2}\lesssim b_{1}^{2k+2}|\log b_{1}|^{C}\for \ 1\le k\le 2,\label{Psi1}\\
\int_{y\leq2B_{1}}|H^{3}\Psi_{b}|^{2}\lesssim\frac{b_{1}^{8}}{|\log b_{1}|^{2}},\label{Psi2}\\
\int_{y\leq2B_{1}}\frac{1+|\log y|^{2}}{1+y^{12-2i}}|\partial_{y}^{i}\Psi_{b}|^{2}\lesssim b_{1}^{8}|\log b_{1}|^{C}\for\  0\leq i\leq4,\label{Psi3}\\
\int_{y\leq2M}|H^{k}\Psi_{b}|^{2}\lesssim b_{1}^{10}M^{C}\for\  0\leq k\leq3. \label{Psi4}
\end{align}
\end{prop}

\begin{proof}
\emph{\uline{Step 1: Computation of the error.}}\emph{ }Take $T_{1},T_{2}$
to be the solution to 
\begin{equation}
HT_{1}+\Lambda Q=0,\ HT_{2}+T_{1}=0,
\end{equation}
given by (\ref{Hu=f:sol}). Then as $y\to+\infty$, we obtain 
\begin{align}
\text{for}\ 0\le i\le 10,\quad \Lambda^{i}T_{1}&=-4\log y+2-4i+O\left(\frac{|\log y|^{2}}{y^{2}}\right),\label{T1-est}\\
\Lambda^{i}T_{2}&=O(y^{2}\log y).\label{T2-est}
\end{align}
There holds the behavior at $y\to 0$
\begin{align}
    \text{for}\ 0\leq i\leq 10,\quad \Lambda^i T_1=O(y^2),\quad \Lambda^i T_2=O(y^4).
\end{align}
Note that $T_1,T_2$ are independent of $b_1$ and $b_2$. In the following we shall find $S_i$ of order $b_1^i$ for $i=2,3,4$. We expand $Q_b^3$ and rearrange them to the polynomial of $b_1$:
\begin{align}
    Q_b^3=Q^3+3Q^2 \alpha + R_2+R_3+R_4+R
\end{align}
where $R$ is a polynomial of $Q,T_{i},S_{j}$ with $O(b_{1}^{5})$
coefficients and 
\begin{equation}\label{R234}
\begin{cases}
R_{2}  :=3b_{1}^{2}QT_{1}^{2}\\
R_{3}  :=6b_{1}b_{2}T_{1}T_{2}+6b_{1}QT_{1}S_{2}+b_{1}^{3}T_{1}^{3},\\
R_{4}  :=6b_{1}QT_{1}S_{3}+3b_{2}^{2}QT_{2}^{2}+6b_{2}QT_{2}S_{2}+3QS_{2}^{2}+3b_{1}^{2}b_{2}T_{1}^{2}T_{2}+3b_{1}^{2}T_{1}^{2}S_{2}.
\end{cases}
\end{equation}
Now a direct computation leads to 
\begin{equation}
\begin{split}
\Psi_{b} ={} & b_{1}\Lambda Q+[(b_{1})_{s}T_{1}+b_{1}^{2}\Lambda T_{1}+b_{1}HT_{1}]+[(b_{2})_{s}T_{2}+b_{1}b_{2}\Lambda T_{2}+b_{2}HT_{2}]\\
& +\sum_{j=2}^{4}[\partial_{s}S_{j}+b_{1}\Lambda S_{j}+HS_{j}]-(Q_{b}^{3}-Q^{3}-3Q^{2}\alpha) \\
={} & b_{1}^{2}[\Lambda T_{1}-(1+c_{b_1})T_{1}]+HS_{2}-R_2 \\
& b_{1}b_{2}[\Lambda T_{2}-(3+c_{b_1})T_{2}]+HS_{3}+(\partial_{s}S_{2}+b_{1}\Lambda S_{2})-R_3 \\
& + HS_{4}+(\partial_{s}S_{3}+b_{1}\Lambda S_{3})-R_4\\        &+(\partial_{s}S_{4}+b_{1}\Lambda S_{4})-R.
\end{split}
\end{equation}

\emph{\uline{Step 2: Construction of the radiation $\varSigma_{b_1}$.}} We introduce a radiation term to cancel
the 1-growth in $\Lambda T_{1}-T_{1}$ (the $\log y$ growth vanishes
by (\ref{T1-est})). Let 
\begin{equation}
c_{b_1}=\frac{64}{\int\chi_{\frac{B_{0}}{4}}(\Lambda Q)^{2}}=\frac{2}{|\log b_{1}|}\left(1+O\left(\frac{1}{|\log b_{1}|}\right)\right),\label{c_b}
\end{equation}
\begin{equation}
d_{b_1}=c_{b_1}\int\chi_{\frac{B_{0}}{4}}\Gamma\Lambda Q=O\left(\frac{1}{b_{1}|\log b_{1}|}\right).
\end{equation}
Let $\varSigma_{b_1}$ solves 
\begin{equation}
H\varSigma_{b_1}=c_{b_1}\chi_{\frac{B_{0}}{4}}\Lambda Q+d_{b_1}H[(1-\chi_{3B_{0}})\Lambda Q],\label{defSigmab}
\end{equation}
that is, 
\begin{equation}
\varSigma_{b_1}=c_{b_1}\Gamma\int_{0}^{y}\chi_{\frac{B_{0}}{4}}(\Lambda Q)^{2}-c_{b_1}\Lambda Q\int_{0}^{y}\chi_{\frac{B_{0}}{4}}\Gamma\Lambda Q+d_{b_1}(1-\chi_{3B_{0}})\Lambda Q.
\end{equation}
The choice of $c_{b_1}$ and $d_{b_1}$ yields 
\begin{equation}
\text{for}\ 0\le i\le 10,\quad  \Lambda^{i}\varSigma_{b_1}=\begin{cases}
-c_{b_1}\Lambda^{i}T_{1} & \text{for}\ y\leq\frac{B_{0}}{4},\\
64\Lambda^{i}\Gamma & \text{for}\ y\geq6B_{0},
\end{cases}\label{Sigmab1}
\end{equation}
and for $\frac{B_{0}}{4}\leq y\leq6B_{0}$,
{
\begin{align}
\begin{split}
\varSigma_{b_1}=&\ c_{b_1}\left(\frac{1}{16}+O(\frac{\log y}{y^{2}})\right)\int_{0}^{y}\chi_{\frac{B_{0}}{4}}(\Lambda Q)^{2}+c_{b_1}O(\frac{1}{y^{2}})O(y^{2})\\
=&\ 4\frac{\int_0^y \chi_{B_0/4}(\Lambda Q)^2}{\int \chi_{B_0/4}(\Lambda Q)^2}+O\left(\frac{1}{|\log b_{1}|}\right),
\end{split}\label{Sigmab2}\\[1em]
\begin{split}
    \Lambda \varSigma_{b_1}=&\ c_{b_1}\Gamma\int_{0}^{y}\chi_{\frac{B_{0}}{4}}(\Lambda Q)^{2}-c_{b_1}\Lambda Q\int_{0}^{y}\chi_{\frac{B_{0}}{4}}\Gamma\Lambda Q+d_{b_1}\Lambda[(1-\chi_{3B_{0}})\Lambda Q].\\ =&\ c_{b_1}\left(\frac{1}{16}+O(\frac{\log y}{y^{2}})\right)\int_{0}^{y}\chi_{\frac{B_{0}}{4}}(\Lambda Q)^{2}+c_{b_1}O(\frac{1}{y^{2}})O(y^{2})\\
=&\ 4\frac{\int_0^y \chi_{B_0/4}(\Lambda Q)^2}{\int \chi_{B_0/4}(\Lambda Q)^2}+O\left(\frac{1}{|\log b_{1}|}\right).
\end{split}
\end{align}
Similarly one can establish,
\begin{align}
    \text{for}\ 0\leq i\leq 10,\quad \Lambda^{i} \varSigma_{b_1}=4\frac{\int_0^y \chi_{B_0/4}(\Lambda Q)^2}{\int \chi_{B_0/4}(\Lambda Q)^2}+O\left(\frac{1}{|\log b_{1}|}\right).
\end{align}
}
We need also to bound $b_{1}^{j}\partial_{b_{1}}^{j}\Lambda^{i}\varSigma_{b_1}$.
Simple calculation reveals that
\[
\text{for}\ 1\leq j\leq 10,\quad  |b_{1}^{j}\partial_{b_{1}}^{j}\chi_{\frac{B_{0}}{4}}|\lesssim\1_{\frac{B_{0}}{4}\leq y\leq\frac{B_{0}}{2}},\quad |b_{1}^{j}\partial_{b_{1}}^{j}\chi_{3B_{0}}|\lesssim\1_{3B_0\leq y\leq 6B_0}.
\]
Since $\Lambda Q$ is independent of $b_1$, then 
\begin{align*}
\left|b_{1}^{j}\partial_{b_{1}}^{j}\int_{0}^{y}\chi_{\frac{B_{0}}{4}}(\Lambda Q)^{2}\right| & \lesssim\int_{\frac{B_{0}}{4}}^{\frac{B_{0}}{2}}(\Lambda Q)^{2}\cdotp\1_{y\geq\frac{B_{0}}{4}}\lesssim\1_{y\geq\frac{B_{0}}{4}},\\[1em]
\left|b_{1}^{j}\partial_{b_{1}}^{j}\int_{0}^{y}\chi_{\frac{B_{0}}{4}}\Gamma\Lambda Q\right| & \lesssim\int_{\frac{B_{0}}{4}}^{\frac{B_{0}}{2}}|\Gamma\Lambda Q|\cdotp\1_{y\geq\frac{B_{0}}{4}} \lesssim\frac{1}{b_{1}}\1_{y\geq\frac{B_{0}}{4}}
\end{align*}
which in particular yield
\[
\text{for}\ 1\leq j\leq 10,\quad \left|b_1^j\partial_{b_1}^jc_{b_1}\right|\lesssim \frac{1}{|\log b_1|^2},\quad |b_1^j\partial_{b_1}^jd_{b_1}|\lesssim \frac{1}{b_1|\log b_1|}.
\]
Taking these into the definition of $\varSigma_{b_1}$, we conclude 
\begin{align}
\text{for}\ 0\leq i\leq 10,\ 1\le j\le 10, \quad |b_{1}^{j}\partial_{b_{1}}^{j}\Lambda^{i}\varSigma_{b_1}| & \lesssim\frac{1}{|\log b_{1}|}\1_{y\leq6B_{0}}.\label{Sigmab3}
\end{align}

\emph{\uline{Step 3: Construction of $S_{2}$.}} Let 
\begin{equation}
\Theta_{2}:=\Lambda T_{1}-T_{1}+\varSigma_{b_1},\label{defTheta2}
\end{equation}
then by (\ref{T1-est}), (\ref{Gamma}), (\ref{Sigmab1}) and (\ref{Sigmab2}), we obtain 
\begin{gather*}
\Lambda^{i}\Theta_{2} = O\left(\frac{|\log y|^{2}}{y^{2}}\right) \for\ y\geq6B_{0},\\[1ex]
\Lambda^{i}\Theta_{2} = O(1) \for\ y\leq\sqrt{B_{0}},\\[1ex]
\Lambda^{i}\Theta_{2} = 4\frac{\int_0^y \chi_{B_0/4}(\Lambda Q)^2}{\int \chi_{B_0/4}(\Lambda Q)^2}-4+O\left(\frac{1}{|\log b_{1}|}\right)+O\left(\frac{|\log y|^{2}}{y^{2}}\right) \for\ \sqrt{B_{0}}\leq y\leq6B_{0}.
\end{gather*}
Combining the above estimates, we have 
\begin{equation}
\text{for}\ 0\leq i\leq 10,\quad  |\Lambda^{i}\Theta_{2}|\lesssim\1_{y\leq1}+\left(\frac{1+|\log \sqrt{b_{1}}y|}{|\log b_{1}|}\right)\1_{1\leq y\leq6B_{0}}+\frac{|\log y|^{2}}{y^{2}}\1_{y\geq6B_{0}}.\label{Theta2-1}
\end{equation}
where we used 
\begin{align*}
\frac{|\log y|^{2}}{y^{2}}\lesssim1\lesssim\frac{|\log \sqrt{b_{1}}y|}{|\log b_{1}|}\for\ 1\leq y\leq\sqrt{B_{0}},\\
\frac{|\log y|^{2}}{y^{2}}\lesssim {\sqrt{b_{1}}}{|\log b_{1}|^{2}}\lesssim\frac{1}{|\log b_{1}|}\for\ \sqrt{B_{0}}\leq y\leq6B_{0},&&.
\end{align*}
Since $T_1$ is independent of $b_1$, then $b_{1}^{j}\partial_{b_{1}}^{j}\Lambda^{i}\Theta_{2}=b_{1}^{j}\partial_{b_{1}}^{j}\Lambda^{i}\varSigma_{b_1}$. Together with (\ref{Sigmab3}) we get 
\begin{equation}
|b_{1}^{j}\partial_{b_{1}}^{j}\Lambda^{i}\Theta_{2}|\lesssim\1_{y\leq1}+\left(\frac{1+|\log \sqrt{b_{1}}y|}{|\log b_{1}|}\right)\1_{1\leq y\leq6B_{0}}+\frac{|\log y|^{2}}{y^{2}}\1_{y\geq6B_{0}}\label{Theta2}
\end{equation}
for $0\leq i,j\leq 10.$


Now let $S_{2}$ be the solution to 
\begin{equation}
b_{1}^{2}\Theta_{2}+HS_{2}-R_{2}=0.\label{defS2}
\end{equation}
{ Recall that  $R_{2}=3b_1^2QT_1^2$. Since $H$ commutes with multiplication of $b_1$, then $S_2=b_1^2\tilde S_2$ where $\tilde S_2$ solves $H\tilde S_2=3QT_1^2-\Theta_2$. The following claim follows from simple calculus.

\begin{claim} Suppose $H^{-1}f$ defined in \eqref{Hu=f:sol}. One has 
\begin{align}
    H^{-1}\Theta_2\lesssim&\ \1_{y\leq 1}+y^2\left(\frac{1+|\log \sqrt{b_{1}}y|}{|\log b_{1}|}\right)\1_{1\leq y\leq6B_{0}}+\frac{1}{b_1|\log b_1|}\1_{y\geq 6B_0},\\[1em]
    H^{-1}[QT_1^2]\lesssim&\ H^{-1}[\1_{y\leq 1}+y^{-2}|\log y|^2\1_{y\geq 1}]\lesssim \1_{y\geq 1}+|\log y|^3\1_{y\geq 1}.
\end{align}
\end{claim}
Above estimates for $\Theta_{2}$ and $R_{2}$ imply 
\begin{equation}
|\tilde S_{2}|\lesssim \1_{y\leq1}+\frac{1}{b_{1}|\log b_{1}|}\1_{y\geq1}\for\ y\le2B_{1}.
\end{equation}
and the rough bound 
\begin{equation}
|\tilde S_{2}|\lesssim 1+y^{2},\for\ y\le2B_{1}.
\end{equation}}
In general, taking $\Lambda^{i}$
operation on (\ref{defS2}) and applying the commutator 
\begin{equation}
[H,\Lambda]=H\Lambda-\Lambda H=2H+(V+\Lambda V),\label{commutator}
\end{equation}
we obtain inductively similar estimates of $\Lambda^i S_2$ for $1\leq i\leq 10$. Also, taking $\partial_{b_{1}}$ operation and (since it commutes with $H$) using (\ref{Theta2}),
then one has 
\begin{align}
\text{for}\ 0\leq i,j\leq 10,\ y\leq2B_{1},\quad  |b_{1}^{j}\partial_{b_{1}}^{j}\Lambda^{i}S_{2}| & \lesssim b_{1}^{2}(\1_{y\leq1}+\frac{1}{b_{1}|\log b_{1}|}\1_{y\geq1}),\label{S2}\\
|b_{1}^{j}\partial_{b_{1}}^{j}\Lambda^{i}S_{2}| & \lesssim b_{1}^{2}(1+y^{2}).\label{roughS2}
\end{align}
Note that $\partial_{b_{2}}S_{2}=0$.

\emph{\uline{Step 4: Construction of $S_{3}$.}} 
Here we use $H^{-1}\varSigma_{b_1}$ to cancel the leading order growth in $\Lambda T_{2}-3T_{2}$. Define 
\begin{equation}
\Theta_{3}:=\Lambda T_{2}-3T_{2}-H^{-1}\varSigma_{b_1}.
\end{equation}
then using (\ref{commutator}) and $HT_{2}=-T_{1}$ we get
\begin{equation}
H\Theta_{3}=-\Lambda T_{1}-2T_{1}+(V+\Lambda V)T_{2}+3T_{1}-\varSigma_{b_1}=-\Theta_{2}+(V+\Lambda V)T_{2},\label{HTheta3}
\end{equation}
hence with bounds (\ref{Theta2}) and (\ref{T2-est}) we derive 
\begin{equation}
\text{for}\ 0\leq i,j,k\leq 10,\ y\leq2B_{1},\quad  |b_{1}^{j+2k}\partial_{b_{1}}^{j}\partial_{b_{2}}^{k}\Lambda^{i}\Theta_{3}|\lesssim\1_{y\leq1}+\frac{1}{b_{1}|\log b_{1}|}\1_{y\geq1}.\label{Theta3}
\end{equation}
Next turn to $R_{3}$. From (\ref{T1-est}), \eqref{T2-est}, (\ref{roughS2}) and \textit{a priori} bound
$|b_{2}|\lesssim b_{1}^{2}$, we get 
\begin{equation}
\text{for}\  0\leq i,j,k\leq 10,\ y\leq2B_{1},\quad  |b_{1}^{j+2k}\partial_{b_{1}}^{j}\partial_{b_{2}}^{k}\Lambda^{i}\ensuremath{R_{3}}|\lesssim b_{1}^{3}(\1_{y\leq1}+|\log y|^{C}\1_{y\geq1})\label{R3}
\end{equation}
Let $S_{3}$ be the solution to 
\begin{equation}
b_{1}b_{2}\Theta_{3}+HS_{3}+[-b_{1}^{2}(1+c_{b_1})+b_{2}]\partial_{b_{1}}S_{2}+b_{1}\Lambda S_{2}-R_{3}=0,\label{defS3}
\end{equation}
then estimates (\ref{Theta3}), (\ref{R3}) and (\ref{S2}) yield 
\begin{align}
\text{for} \ 0\leq i,j,k\leq10,\  y\leq2B_{1},\quad   |b_{1}^{j+2k}\partial_{b_{1}}^{j}\partial_{b_{2}}^{k}\Lambda^{i}S_{3}| & \lesssim b_{1}^{3}\left(\1_{y\leq1}+\frac{y^{2}}{b_{1}|\log b_{1}|}\1_{y\geq1}\right),\label{S3}\\
|b_{1}^{j+2k}\partial_{b_{1}}^{j}\partial_{b_{2}}^{k}\Lambda^{i}S_{3}| & \lesssim b_{1}^{3}(1+y^{4}).\label{roughS3}
\end{align}

\emph{\uline{Step 5: Construction of $S_{4}$.}} 
From (\ref{roughS2}), (\ref{roughS3}), we get 
\begin{equation}
\text{for}\ 0\leq i,j,k\leq 10,\ y\leq2B_{1},\quad |b_{1}^{j+2k}\partial_{b_{1}}^{j}\partial_{b_{2}}^{k}\Lambda^{i}R_{4}|\lesssim b_{1}^{4}(\1_{y\leq1}+y^{2}|\log y|^{C}\1_{y\geq1}).\label{R4}
\end{equation}
Let $S_{4}$ be the solution to 
\begin{equation}
HS_{4}+[-b_{1}^{2}(1+c_{b_1})+b_{2}]\partial_{b_{1}}S_{3}+[-b_{1}b_{2}(3+c_{b_1})]\partial_{b_{2}}S_{3}+b_{1}\Lambda S_{3}-R_{4}=0,\label{defS4}
\end{equation}
then similar to Step 4, we have 
\begin{align}
\text{for}\ 0\leq i,j,k\leq10,\ y\leq2B_{1},\quad  |b_{1}^{j+2k}\partial_{b_{1}}^{j}\partial_{b_{2}}^{k}\Lambda^{i}S_{4}| & \lesssim b_{1}^{4}\left(\1_{y\leq1}+\frac{y^{4}}{b_{1}|\log b_{1}|}\1_{y\geq1}\right),\label{S4}\\
|b_{1}^{j+2k}\partial_{b_{1}}^{j}\partial_{b_{2}}^{k}\Lambda^{i}S_{4}| & \lesssim b_{1}^{4}(1+y^{6}).\label{roughS4}
\end{align}

\emph{\uline{Step 6: Estimation of the error.}} According to our
constructions above and assumption (\ref{roughModeq}), we get

\begin{align}\label{defPsib}
\begin{split}
\Psi_{b}=& -b_{1}^{2}(\varSigma_{b_1}+c_{b_1}T_{1})+b_{1}b_{2}H^{-1}(\varSigma_{b_1}+c_{b_1}T_{1})\\
 & +[-b_{1}^{2}(1+c_{b_1})+b_{2}]\partial_{b_{1}}S_{4}+[-b_{1}b_{2}(3+c_{b_1})]\partial_{b_{2}}S_{4}+b_{1}\Lambda S_{4}\\
 & -R.
 \end{split}
\end{align}
We split it into three parts and estimate as follows:

(i) The first line: Write $\tilde{\varSigma}_{b_1}:=\varSigma_{b_1}+c_{b_1}T_{1}$.
Note that $\tilde{\varSigma}_{b_1}=0$ for $y\leq\frac{B_{0}}{4}$.

By (\ref{defSigmab}), we estimate for $\frac{B_{0}}{4}\le y\leq2B_{1}$
\begin{gather*}
|H^{k}\tilde{\varSigma}_{b_1}|\lesssim\frac{1}{|\log b_{1}|y^{2k}}\for \  1\leq k\leq3,\\
|\partial_{y}^{i}\tilde{\varSigma}_{b_1}|\lesssim y^{-i},\quad  |\partial_{y}^{i}H^{-1}\tilde{\varSigma}_{b_1}|\lesssim y^{2-i}\for\  0\leq i\leq10,
\end{gather*}
hence we conclude 
\begin{gather}
\int_{y\leq2B_{1}}|H^{k}\tilde{\varSigma}_{b_1}|^{2}\lesssim\int_{\frac{B_{0}}{4}\le y\le2B_{1}}\frac{1}{y^{4k}}\lesssim b_{1}^{2k-2}|\log b_{1}|^{C}\for\ k=-1,0,1,\label{(i)}\\
\int_{y\leq2B_{1}}|H^{k}\tilde{\varSigma}_{b_1}|^{2}\lesssim\frac{1}{|\log b_{1}|^{2}}\int_{\frac{B_{0}}{4}\le y\le2B_{1}}\frac{1}{y^{4k}}\lesssim\frac{b_{1}^{2k-2}}{|\log b_{1}|^{2}}\for\ k=2,3.
\end{gather}

(ii) The second line: By the rough bound (\ref{roughS4}), we have
\begin{equation}
\begin{split}
&\int_{y\leq2B_{1}}|H^{k}\Lambda S_{4}|^{2}+|H^{k}(b_{1}\partial_{b_{1}}S_{4})|^{2}+|H^{k}(b_{1}^{2}\partial_{b_{2}}S_{4})|^{2}\\
\lesssim{} & b_{1}^{8}\int_{y\leq2B_{1}}(1+y^{6-2k})^{2}\lesssim b_{1}^{2k}|\log b_{1}|^{C}\for\  k=1,2.
\end{split}
\end{equation}
The crucial $H^{3}$ level bound requires more effort. Apply operator
$H$ twice to (\ref{defS4}) and use (\ref{commutator}) , then we
find 
\begin{align*}
H^{3}S_{4} & =O(b_{1})(H^{2}S_{3}+\Lambda H^{2}S_{3}+b_{1}\partial_{b_{1}}H^{2}S_{3}+b_{1}^{2}\partial_{b_{2}}H^{2}S_{3})+O(b_{1}^{4})(\1_{y\leq1}+\frac{|\log y|^{C}}{y^{2}}\1_{y\geq1}).
\end{align*}
Again, apply $H$ to (\ref{defS3}) and use (\ref{commutator}) 
\[
H^{2}S_{3}=O(b_{1}^{3})(\Theta_{2}+\Lambda\Theta_{2}+b_{1}\partial_{b_{1}}\Theta_{2})+O(b_{1}^{4})\left(\1_{y\leq1}+\frac{|\log y|^{C}}{y^{2}}\1_{y\geq1}\right).
\]
Now the bound (\ref{Theta2}) for $\Theta_{2}$ implies 
\[
|H^{3}S_{4}|\lesssim b_{1}^{4}\left(\1_{y\leq1}+\left(\frac{1+|\log \sqrt{b_{1}}y|}{|\log b_{1}|}\right)\1_{1\leq y\leq6B_{0}}+\frac{|\log y|^{C}}{y^{2}}\1_{y\geq6B_{0}}\right)
\]
thus 
\begin{align*}
\int_{y\leq2B_{1}}|H^{3}S_{4}|^{2} & \lesssim b_{1}^{8}\left(\int_{y\leq1}1+\int_{1\leq y\leq6B_{0}}\frac{(1+|\log \sqrt{b_{1}}y|)^{2}}{|\log b_{1}|^{2}}+\int_{6B_{0}\leq y\leq2B_{1}}\frac{|\log y|^{C}}{y^{4}}\right)\\
&\lesssim\frac{b_{1}^{6}}{|\log b_{1}|^{2}}.
\end{align*}
Similarly we have 
\begin{equation}
\int_{y\leq2B_{1}}|H^{3}\Lambda S_{4}|^{2}+|H^{3}(b_{1}\partial_{b_{1}}S_{4})|^{2}+|H^{3}(b_{1}^{2}\partial_{b_{2}}S_{4})|^{2}\lesssim\frac{b_{1}^{6}}{|\log b_{1}|^{2}}.
\end{equation}

(iii) The third line: From rough bounds (\ref{roughS2}), (\ref{roughS3}), (\ref{roughS4})
for $S_{j}$ we derive
\begin{align*}
|\partial_{y}^{i}R| & \lesssim b_{1}^{5}\1_{y\leq1}+\sum_{j=5}^{12}b_{1}^{j}y^{2j-6-i}(1+|\log y|^{C})\1_{y\geq1}\\
&\lesssim b_{1}^{5}(\1_{y\leq1}+y^{4-i}|\log b_{1}|^{C}\1_{y\geq1})\ \text{for}\ y\leq2B_{1}.
\end{align*}
Hence 
\begin{align}
\begin{split}
\int_{y\leq2B_{1}}|H^{k}R|^{2}&\lesssim b_{1}^{10}\int_{y\leq1}1+b_{1}^{10}|\log b_{1}|^{C}\int_{1\leq y\leq2B_{1}}y^{8-4k}\\[1ex]
&\lesssim b_{1}^{2k+4}|\log b_{1}|^{C}\ \text{for}\ 1\leq k\leq3.\label{(iii)}
\end{split}
\end{align}

The bounds (\ref{(i)})-(\ref{(iii)}) together yield (\ref{Psi1}) and (\ref{Psi2}).
(\ref{Psi3}) can be proved in similar way. (\ref{Psi4}) follows
from $\tilde{\varSigma}_{b_1}=0$ for $y\le2M\le B_{0}/4$. 
\end{proof}

\subsection{Localization}

The approximate solution $Q_{b}$ we constructed above stays close
to $Q$ only in the parabolic zone $y\leq2B_{1}$, and certain localization
is needed to avoid the growth at infinity. 
\begin{prop}
(Localization) Under the assumptions of Proposition $\text{\ref{prop:approximate}}$,
assume further that 
\begin{equation}
|(b_{1})_{s}|\lesssim b_{1}^{2}.\label{apriori2}
\end{equation}
Let $\tilde{Q}_{b}:=Q+\tilde{\alpha}$, $\tilde{\alpha}:=\chi_{B_{1}}\alpha.$
Then 
\begin{equation}
\partial_{s}\tilde{Q}_{b}-\frac{\lambda_{s}}{\lambda}\Lambda\tilde{Q}_{b}-\Delta\tilde{Q}_{b}-\tilde{Q}_{b}^{3}=\tilde{\Psi}_{b}+\Mod,\label{Qb'eq}
\end{equation}
where 
\begin{equation}
\begin{split}
\Mod:=&-\left(\frac{\lambda_{s}}{\lambda}+b_{1}\right)\Lambda\tilde{Q}_{b}+[(b_{1})_{s}+b_{1}^{2}(1+c_{b_1})-b_{2}]\left[\tilde{T}_{1}+\chi_{B_{1}}\sum_{j=2}^{4}\frac{\partial S_{j}}{\partial b_{1}}\right]\\
&+[(b_{2})_{s}+b_{1}b_{2}(3+c_{b_1})]\left[\tilde{T}_{2}+\chi_{B_{1}}\sum_{j=3}^{4}\frac{\partial S_{j}}{\partial b_{2}}\right]\label{defMod}
\end{split}
\end{equation}
and $\tilde{\Psi}_{b}$ satisfies 
\begin{gather}
\int|H^{k}\tilde{\Psi}_{b}|^{2}\lesssim b_{1}^{2k+2}|\log b_{1}|^{C}\ \for\ k=1,2,\label{Psi1'}\\
\int|H^{3}\tilde{\Psi}_{b}|^{2}\lesssim\frac{b_{1}^{8}}{|\log b_{1}|^{2}},\label{Psi2'}\\
\int\frac{1+|\log y|^{2}}{1+y^{12-2i}}|\partial_{y}^{i}\tilde{\Psi}_{b}|^{2}\lesssim b_{1}^{8}|\log b_{1}|^{C}\ \for\ 0\leq i\leq4,\label{Psi3'}\\
\int_{y\leq2M}|H^{k}\tilde{\Psi}_{b}|^{2}\lesssim b_{1}^{10}M^{C}\ \for\ 0\leq k\leq3.\label{Psi4'}
\end{gather}
\end{prop}

\begin{proof}
From localization we compute 
\begin{align*}
\tilde{\Psi}_{b}=&\ \chi_{B_{1}}\Psi_{b}+b_{1}(1-\chi_{B_{1}})\Lambda Q+\alpha\partial_{s}\chi_{B_{1}}-\alpha\Delta\chi_{B_{1}}-2\partial_{y}\alpha\partial_{y}\chi_{B_{1}}\\
&\ -[(Q+\chi_{B_{1}}\alpha)^{3}-Q^{3}]+\chi_{B_{1}}[(Q+\alpha)^{3}-Q^{3}].
\end{align*}
We estimate the $H^{3}$ level bound (\ref{Psi2'}) term by term,
as follows.

(i) Start with $\chi_{B_{1}}\Psi_{b}$. From (\ref{defPsib}) we estimate
for $0\leq i\leq6,\ B_{1}\le y\leq2B_{1}$ 
\[
|\partial_{y}^{i}\Psi_{b}|\lesssim b_{1}^{2}y^{-i}+b_{1}^{3}y^{2-i}+\frac{b_{1}^{4}y^{4-i}}{|\log b_{1}|}+b_{1}^{5}y^{4-i}|\log b_{1}|^{C}\lesssim\frac{b_{1}^{4}y^{4-i}}{|\log b_{1}|}.
\]
Together with (\ref{Psi2}) we get 
\[
\int|H^{3}(\chi_{B_{1}}\Psi_{b})|^{2}\lesssim\int_{y\leq B_{1}}|H^{3}\Psi_{b}|^{2}+\frac{b_{1}^{8}}{|\log b_{1}|^{2}}\int_{B_{1}\leq y\leq2B_{1}}\frac{1}{y^{4}}\lesssim\frac{b_{1}^{8}}{|\log b_{1}|^{2}}.
\]

(ii) The second term is easily controlled by 
\[
\int\left|H^{3}[(1-\chi_{B_{1}})\Lambda Q]\right|^{2}\lesssim\int_{y\ge B_{1}}\frac{1}{y^{16}}\lesssim\frac{b_{1}^{6}}{|\log b_{1}|^{2}}.
\]

(iii) For the third term $\alpha\partial_{s}\chi_{B_{1}}$, with \textit{a priori} bound $|(b_{1})_{s}|\lesssim b_{1}^{2}$, we have
\[
0\leq i\leq 10,\quad |\partial_{y}^{i}\partial_{s}\chi_{B_{1}}|=|(b_{1})_{s}\partial_{b_{1}}\partial_{y}^{i}\chi_{B_{1}}|\lesssim\frac{b_{1}}{y^{i}}\1_{B_{1}\leq y\leq2B_{1}},
\]
and our calculations in the previous subsection imply 
\[
|\partial_{y}^{i}\alpha|\lesssim b_{1}^{2}y^{2-i}\log y\ \text{for}\ B_{1}\le y\leq2B_{1}.
\]
So we have 
\[
\int|H^{3}(\alpha\partial_{s}\chi_{B_{1}})|^{2}\lesssim b_{1}^{4}\int_{B_{1}\leq y\leq2B_{1}}\frac{|\log y|^{2}}{y^{8}}\lesssim\frac{b_{1}^{8}}{|\log b_{1}|^{2}}.
\]

(iv) The next two terms are bounded by 
\[
\int|H^{3}(\alpha\Delta\chi_{B_{1}}+2\partial_{y}\alpha\partial_{y}\chi_{B_{1}})|^{2}\lesssim b_{1}^{4}\int_{B_{1}\leq y\leq2B_{1}}\frac{|\log y|^{2}}{y^{8}}\lesssim\frac{b_{1}^{8}}{|\log b_{1}|^{2}}.
\]

(iv) Finally, note that 
\[
-[(Q+\chi_{B_{1}}\alpha)^{3}-Q^{3}]+\chi_{B_{1}}[(Q+\alpha)^{3}-Q^{3}]=3(\chi_{B_{1}}-\chi_{B_{1}}^{2})Q\alpha^{2}+(\chi_{B_{1}}-\chi_{B_{1}}^{3})\alpha^{3},
\]
and 
\begin{align*}
&\int\left|H^{3}[3(\chi_{B_{1}}-\chi_{B_{1}}^{2})Q\alpha^{2}+(\chi_{B_{1}}-\chi_{B_{1}}^{3})\alpha^{3}]\right|^{2}\\
\lesssim{}&\int_{B_{1}\leq y\leq2B_{1}}\left|\frac{1}{y^{6}}[b_{1}^{4}y^{2}(\log y)^{2}+b_{1}^{6}y^{6}(\log y)^{3}]\right|^{2}\\
\lesssim{}& b_{1}^{10}|\log b_{1}|^{C}\lesssim\frac{b_{1}^{8}}{|\log b_{1}|^{2}}.
\end{align*}
This concludes the proof of (\ref{Psi2'}). Proofs for the other three
are similar(and simpler). 
\end{proof}

\subsection{Dynamical system for $b=(b_{1},b_{2})$}
\label{Sec:Dynamics for b}
In the first subsection, we have seen the importance of the (modulation)
assumption (\ref{roughModeq}). Thus $b=(b_{1,}b_{2})$ should approximately
satisfies
\[
(b_{k})_{s}+(2k-1+\frac{2}{\log s})b_{1}b_{k}-b_{k+1}=0,\quad k=1,2,\quad b_{3}=0.
\]
Indeed, this equation has an approximate solution 
\begin{equation}\label{def:b^e}
b_{1}^{e}=\frac{2}{3s}-\frac{4}{9s\log s},\quad b_{2}^{e}=-\frac{2}{9s^{2}}+\frac{20}{27s^{2}\log s},\quad b_{3}^{e}=0
\end{equation}
in the sense that 
\begin{equation}
(b_{k}^{e})_{s}+\left(2k-1+\frac{2}{\log s}\right)b_{1}^{e}b_{k}^{e}-b_{k+1}^{e}=O\left(\frac{1}{s^{k+1}(\log s)^{2}}\right),\quad k=1,2.
\end{equation}
The proof is by direct calculation. Now we look for $b$ near this
approximate solution. 
\begin{prop}\label{prop:dynamics for b}Let 
\begin{equation}
b_{k}=b_{k}^{e}+\frac{U_{k}}{s^{k}(\log s)^{\frac{5}{4}}},\quad k=1,2,\quad b_{3}=0\label{decomb}
\end{equation}
and $U=\begin{bmatrix}U_{1}\\
U_{2}
\end{bmatrix}$, $b_{k}^{e}$ in \eqref{def:b^e}. Then we have 
\begin{equation}
\begin{split}
&(b_{k})_{s}+\left(2k-1+\frac{2}{\log s}\right)b_{1}b_{k}-b_{k+1}\\
={}&\frac{1}{s^{k+1}(\log s)^{\frac{5}{4}}}\left[s(U_{k})_{s}-(AU)_{k}+O\left(\frac{1}{\sqrt{\log s}}+\frac{|U|+|U|^{2}}{\log s}\right)\right]\label{dynamicb}
\end{split}
\end{equation}
where $A=\begin{bmatrix}-\frac{1}{3} & 1\\[1ex]
\frac{2}{3} & 0
\end{bmatrix}=P^{-1}D_{A}P$ with $P=\begin{bmatrix}1 & -1\\
2 & 3
\end{bmatrix},\ D_{A}=\begin{bmatrix}-1\\
 & \frac{2}{3}
\end{bmatrix}$. 
\end{prop}

\begin{proof}
In fact, direct computation yields 
\begin{equation}
\begin{split}
& (b_{1})_{s}+(1+\frac{2}{\log s})b_{1}^{2}-b_{2}\\
={}&\frac{1}{s^{2}(\log s)^{\frac{5}{4}}}\left[s(U_{1})_{s}-U_{1}+O(\frac{|U|}{\log s})\right]\\
&+\frac{4}{3}U_{1}-U_{2}+O\left(\frac{|U|+|U|^{2}}{\log s}\right)+O\left(\frac{1}{s^{2}|\log s|^{2}}\right)
\end{split}
\end{equation}
and 
\begin{equation}
\begin{split}
&(b_{2})_{s}+(3+\frac{2}{\log s})b_{1}b_{2}\\
={}&\frac{1}{s^{3}(\log s)^{\frac{5}{4}}}\left[s(U_{2})_{s}-2U_{2}+O(\frac{|U|}{\log s})\right]\\
& +3\left(-\frac{2}{9}U_{1}+\frac{2}{3}U_{2}\right)+O\left(\frac{|U|+|U|^{2}}{\log s}\right)+O\left(\frac{1}{s^{3}|\log s|^{2}}\right).
\end{split}
\end{equation}
which can be arranged to (\ref{dynamicb}). The diagonalization of
$A$ is simple linear algebra. 
\end{proof}

\section{THE TRAPPED REGIME}
\label{THE TRAPPED REGIME}
From now on, we assume that the initial data $v_{0}\in\dot{H}^{1}\cap \dot{H}^{6}$.
From standard local well posedness theory, (\ref{eq1}) has a solution
$v\in C([0,T),\dot{H}^{1}\cap \dot{H}^{6})$ with lifetime $T=T(v_{0})\le+\infty$.

In this section we describe our choice of initial data and design
a bootstrap regime to control the behavior of the corresponding solution.
Our main analysis is on the Lyapounov monotonicity in subsection 3.3.

\subsection{Modulation theory}

We first try to decompose the solution as 
\begin{equation}
v(t,\cdotp)=(\tilde{Q}_{b(t)}+\varepsilon)_{\lambda(t)},\quad \text{or}\ u=\tilde{Q}_{b}+\varepsilon.\label{decom}
\end{equation}
where $\varepsilon(s,y)$ satisfies the orthogonality conditions 
\begin{equation}
(\varepsilon,H^{k}\Phi_{M})=0\ \text{for}\ 0\le k\le2.\label{orthoepsilon}
\end{equation}
Here $\Phi_{M}$ is a substitute for $\Lambda Q$ supported on $y\le2M$
(this is necessary since $\Lambda Q\notin L^{2}$), as described below.

Given $M>0$ large enough, define 
\begin{equation}
\Phi_{M}:=\chi_{M}\Lambda Q+c_{M,1}H(\chi_{M}\Lambda Q)+c_{M,2}H^{2}(\chi_{M}\Lambda Q)\label{defPhi}
\end{equation}
where 
\[
c_{M,1}=\frac{(\chi_{M}\Lambda Q,T_{1})}{(\chi_{M}\Lambda Q,\Lambda Q)}=O(M^{2}),\ c_{M,2}=\frac{-(\chi_{M}\Lambda Q,T_{2})+c_{M,1}(\chi_{M}\Lambda Q,T_{1})}{(\chi_{M}\Lambda Q,\Lambda Q)}=O(M^{4})
\]
are chosen to ensure the cancellation 
\begin{equation}
(\Phi_{M},T_{i})=0,\quad i=1,2\label{orthoPhi}
\end{equation}
and non-degeneracy 
\begin{equation}
(\Phi_{M},\Lambda Q)=(\chi_{M}\Lambda Q,\Lambda Q)=64\log M(1+o_{M\to+\infty}(1)),
\end{equation}
\begin{equation}
||\Phi_{M}||_{L^{2}}^{2}\lesssim\int|\chi_{M}\Lambda Q|^{2}+c_{M,1}^{2}\int|H(\chi_{M}\Lambda Q)|^{2}+c_{M,2}^{2}\int|H^{2}(\chi_{M}\Lambda Q)|^{2}\lesssim \log M.
\end{equation}
Now at the point $\lambda=1,\ b=(b_{1},b_{2})=(0,0)$, we have 
\[
(\partial_{\lambda}(\tilde{Q}_{b})_{\lambda},\partial_{b_{1}}(\tilde{Q}_{b})_{\lambda},\partial_{b_{2}}(\tilde{Q}_{b})_{\lambda})=(\Lambda Q,T_{1},T_{2}),
\]
hence the non-degeneracy of Jacobian 
\begin{align*}
\left|\left(\frac{\partial}{\partial(\lambda,b_j)}(\tilde Q_b)_\lambda,H^i\Phi_M\right)_{\substack{1\leq j\leq 2,\\0\leq i\leq 2}}\right|_{\lambda=1,b=0}& =\begin{vmatrix}(\Lambda Q,\Phi_{M}) & 0 & 0\\
0 & -(\Lambda Q,\Phi_{M}) & 0\\
0 & 0 & (\Lambda Q,\Phi_{M})
\end{vmatrix}\neq0,
\end{align*}
Now we introduce some notations:

(i) The energy norm 
\begin{equation}
\varXi_{1}(t):=\int|\partial_{y}\varepsilon|^{2}
\end{equation}
and higher Sobolev norms 
\begin{equation}
\varXi_{2k}(t):=\int|\varepsilon_{2k}|^{2},\quad 1\leq k\leq3\ \text{with}\ \varepsilon_{2k}:=H^{k}\varepsilon.
\end{equation}

(ii) The unstable models 
\begin{equation}
V(t)=\begin{bmatrix}V_{1}(t)\\
V_{2}(t)
\end{bmatrix}=PU,\quad  \text{where }P,U\ \text{are from Proposition \ref{prop:dynamics for b}}.
\end{equation}
\begin{equation}
\tau(t)=(\varepsilon(t),\psi),\quad\text{ where }\psi \text{ is from }\eqref{def:psi}.\label{deftau}
\end{equation}

With these preparations, we turn to the construction of initial data.
Set $v_{0}$ in the decomposition form (\ref{decom}) as
\begin{equation}
v_{0}=\tilde{Q}_{b(0)}+\tau(0)\tilde{\psi}
\end{equation}
where $\tilde{\psi}$ satisfies
\begin{equation}
(\tilde{\psi},\psi)=1,\quad (\tilde{\psi},H^{k}\Phi_{M})=0,\quad 0\le k\le2.
\end{equation}
This way, the orthogonal conditions (\ref{orthoepsilon}) are automatically satisfied at $t=0$.

Besides, $b(0)$ is determined by $V(0)$ (or $U(0)$) through (\ref{decomb})
\begin{equation}
b_{k}=b_{k}^{e}+\frac{U_{k}}{s^{k}(\log s)^{\frac{5}{4}}},\quad k=1,2.
\end{equation}
We fix $U_{1}(0)=0$, hence 
\begin{equation}
b_{1}(0)=\frac{2}{3s_{0}}-\frac{4}{9s_{0}}.
\end{equation}
Choose $s_{0}$ large enough and $V(0),\tau(0)$ properly so that
the following bounds hold:

(i) Initial smallness: 
\begin{equation}
0<b_{1}(0)<b^{*}(M)\ll1,\ \varXi_{1}(0)\leq b_{1}(0),\label{initial1}
\end{equation}
\begin{equation}
\varXi_{2}(0)+\varXi_{4}(0)+\varXi_{6}(0)\leq b_{1}(0)^{7}.\label{initial2}
\end{equation}

(ii) Control of the unstable models: 
\begin{equation}
|V_{1}(0)|\leq1,\quad |V_{2}(0)|\leq1,\quad |\tau(0)|\leq\frac{b_{1}(0)^{3+\frac{1}{2}}}{|\log b_{1}(0)|}\label{initial3}
\end{equation}

(iii) Without loss of generality we assume 
\begin{equation}
\lambda(0)=1.
\end{equation}
By Implicit Function Theorem, (\ref{initial1}) ensures that the decomposition
(\ref{decom}) exists and is unique near $t=0$. Moreover, $\lambda,b\in C^{1}$.

Given another large enough universal constant $K>0$, independent
of $M$, the continuity of the flow implies the following proposition.
\begin{prop}
\label{prop:bootstrap} (Bootstrap) There exists some maximal time
$T_{exit}\in[0,T(v_{0})]$, called exit time, such that for all $t\in[0,T_{exit})$
the following bounds hold:

(i) Control of $\varepsilon$: 
\begin{gather}
\varXi_{1}(t)\leq10\sqrt{b_{1}(0)},\label{exit1}\\
\varXi_{2}(t)\leq b_{1}^{\frac{4}{3}}(t)|\log b_{1}(t)|^{K},\\
\varXi_{4}(t)\leq b_{1}^{4}(t)|\log b_{1}(t)|^{K},\\
\varXi_{6}(t)\leq K\frac{b_{1}^{6}(t)}{|\log b_{1}(t)|^{2}}.\label{exit2}
\end{gather}

(ii) Control of unstable models:
\begin{equation}
|V_{1}(t)|\leq2,\quad |V_{2}(t)|\leq2,\quad |\tau(t)|\leq\frac{b_{1}(t)^{3+\frac{1}{2}}}{|\log b_{1}(t)|}.\label{exit3}
\end{equation}
\end{prop}

\begin{rem}
Bound (\ref{exit3}) on $V$ ensures $|b_{2}|\lesssim b_{1}^{2}$,
so (\ref{apriori1}) holds. 
\end{rem}

We now describe bootstrap regime. First, use the control of unstable
models to improve the bounds (\ref{exit1})-(\ref{exit2}). Next,
for $V$ and $\tau$, our primary observation is that the two unstable
directions are in some sense independent of each other, hence a Brouwer
argument works to provide some initial data with $T_{exit}=T(v_{0})$.
The proof of Theorem \ref{thm:main} follows from  Proposition \ref{prop:bootstrap} easily.

In the next two subsections we assume $t\in[0,T_{exit})$ and deduce
some key tools to close the bootstrap.

\subsection{Modulation equations}
\label{Modulation equations}
Bring (\ref{decom}) and (\ref{Qb'eq}) into (\ref{eq2}), we find
\begin{equation}
\partial_{s}\varepsilon-\frac{\lambda_{s}}{\lambda}\Lambda\varepsilon+H\varepsilon=\mathscr{F}=-\tilde{\Psi}_{b}-\Mod+L(\varepsilon)+N(\varepsilon)\label{dynamicepsilon}
\end{equation}
where 
\begin{equation}
L(\varepsilon)=3(\tilde{Q}_{b}-Q^{2})\varepsilon,\quad N(\varepsilon)=3\tilde{Q}_{b}\cdotp\varepsilon^{2}+\varepsilon^{3}.
\end{equation}
We now derive the modulation equations for $b,\lambda$ as a consequence
of the orthogonality conditions (\ref{orthoepsilon}). 
\begin{prop}
(Modulation equations) We have the bounds on the modulation parameters
\begin{equation}
\left|\frac{\lambda_{s}}{\lambda}+b_{1}\right|+|(b_{1})_{s}+b_{1}^{2}(1+c_{b_1})-b_{2}|\lesssim b_{1}^{3+\frac{1}{2}},\label{Modeq1}
\end{equation}
\begin{equation}
|(b_{2})_{s}+b_{1}b_{2}(3+c_{b_1})|\lesssim\frac{1}{\sqrt{\log M}}\left(\sqrt{\varXi_{6}}+\frac{b_{1}^{3}}{|\log b_{1}|}\right),\label{Modeq2}
\end{equation}
with constants independent of $M$ and $K$, as long as the $b^{*}(M)$
in \eqref{apriori1} is small enough. 
\end{prop}

\begin{rem}
(\ref{Modeq1}) shows $|(b_{1})_{s}|\leq b_{1}^{2}$ and hence (\ref{apriori2})
holds. 
\end{rem}

\begin{proof}
Let 
\begin{equation}
D(t):=\left|\frac{\lambda_{s}}{\lambda}+b_{1}\right|+|(b_{1})_{s}+b_{1}^{2}(1+c_{b_1})-b_{2}|+|(b_{2})_{s}+b_{1}b_{2}(3+c_{b_1})|.
\end{equation}

\noindent \emph{\uline{Step 1: Law for $b_{2}$.}} Take the inner product of (\ref{dynamicepsilon})
with $H^{2}\Phi_{M}$. Using (\ref{orthoepsilon}) we get 
\begin{align*}
(\Mod,H^{2}\Phi_{M}) & =-(H^{3}\varepsilon,\Phi_{M})-(H^{2}\tilde{\Psi}_{b},\Phi_{M})+\left(\frac{\lambda_{s}}{\lambda}\Lambda\varepsilon+L(\varepsilon)+N(\varepsilon),H^{2}\Phi_{M}\right).
\end{align*}
On the other hand, by the definition (\ref{defMod}) of $\Mod$, 
\begin{align*}
&(\Mod,H^{2}\Phi_{M})\\
={} & -\left(\frac{\lambda_{s}}{\lambda}+b_{1}\right)(\Lambda\tilde{Q}_{b},H^{2}\Phi_{M})+[(b_{1})_{s}+b_{1}^{2}(1+c_{b_1})-b_{2}]\left(\tilde{T}_{1}+\chi_{B_{1}}\sum_{j=2}^{4}\frac{\partial S_{j}}{\partial b_{1}},H^{2}\Phi_{M}\right)\\
 & +[(b_{2})_{s}+b_{1}b_{2}(3+c_{b_1})]\left(\tilde{T}_{2}+\chi_{B_{1}}\sum_{j=3}^{4}\frac{\partial S_{j}}{\partial b_{2}},H^{2}\Phi_{M}\right)\\
= {}&[(b_{1})_{s}+b_{1}^{2}(1+c_{b_1})-b_{2}]\left(\sum_{j=2}^{4}\frac{\partial S_{j}}{\partial b_{1}},H^{2}\Phi_{M}\right)\\
 & +[(b_{2})_{s}+b_{1}b_{2}(3+c_{b_1})]\left(\sum_{j=3}^{4}\frac{\partial S_{j}}{\partial b_{2}},H^{2}\Phi_{M}\right)+[(b_{2})_{s}+b_{1}b_{2}(3+c_{b_1})](\Lambda Q,\Phi_{M})
\end{align*}
where we used $H^{2}\Lambda\tilde{Q}_{b}=H^{2}\tilde{T}_{1}=0$ and
$H^{2}\tilde{T}_{2}=\Lambda Q$ for $y\leq2M$.

We now estimate these terms respectively. By Cauchy-Schwarz inequality and (\ref{Psi4'})
\[
|(H^{3}\varepsilon,\Phi_{M})|\lesssim||H^{3}\varepsilon||_{L^{2}}||\Phi_{M}||_{L^{2}}\lesssim\sqrt{\log M}\sqrt{\varXi_{6}},
\]
\[
|(H^{2}\tilde{\Psi}_{b},\Phi_{M})|\lesssim||H^{2}\tilde{\Psi}_{b}||_{L^{2}(y\leq2M)}||\Phi_{M}||_{L^{2}}\lesssim b_{1}^{5}M^{C}.
\]
By interpolation bounds in the Appendix B and \eqref{exit2}, we obtain
\begin{align*}
\left|\left(\frac{\lambda_{s}}{\lambda}\Lambda\varepsilon,H^{2}\Phi_{M}\right)\right| & \lesssim(|D(t)|+b_{1})\cdotp C(M)\sqrt{\varXi_{6}}\lesssim b_{1}M^{C}|D(t)|+\sqrt{\log M}\sqrt{\varXi_{6}},
\end{align*}
\[
|(L(\varepsilon)+N(\varepsilon),H^{2}\Phi_{M})|\lesssim b_{1}C(M)\sqrt{\varXi_{6}}\lesssim\sqrt{\log M}\sqrt{\varXi_{6}}.
\]
Using rough bound of $S_{j}$ in \eqref{roughS2}, \eqref{roughS3}, \eqref{roughS4}, one has 
\begin{align*}
\left|(b_{1})_{s}+b_{1}^{2}(1+c_{b_1})-b_{2})\left(\sum_{j=2}^{4}\frac{\partial S_{j}}{\partial b_{1}},H^{2}\Phi_{M}\right)+[(b_{2})_{s}+b_{1}b_{2}(3+c_{b_1})]\left(\sum_{j=3}^{4}\frac{\partial S_{j}}{\partial b_{2}},H^{2}\Phi_{M}\right)\right|\\[1em]
\lesssim b_{1}M^{C}|D(t)|.
\end{align*}
So we conclude that
\begin{equation}\label{step1}
\begin{split}
|(b_{2})_{s}+b_{1}b_{2}(3+c_{b_1})| & =\frac{1}{|(\Lambda Q,\Phi_{M})|}O\left(\sqrt{\log M}\sqrt{\varXi_{6}}+b_{1}^{5}M^{C}+b_{1}M^{C}|D(t)|\right)\\
 & \lesssim\frac{1}{\sqrt{\log M}}\left(\sqrt{\varXi_{6}}+\frac{b_{1}^{3}}{|\log b_{1}|}\right)+b_{1}M^{C}|D(t)|.
 \end{split}
\end{equation}

\noindent \emph{\uline{Step 2: Law for $b_{1}$ and $\lambda$.}} Similarly, take the inner product of (\ref{dynamicepsilon})
with $H\Phi_{M}$ (note that $(H\varepsilon,H\Phi_{M})=(\varepsilon,H^{2}\Phi_{M})=0$)
: 
\[
(\Mod,H\Phi_{M})=-(H\tilde{\Psi}_{b},\Phi_{M})+\left(\frac{\lambda_{s}}{\lambda}\Lambda\varepsilon+L(\varepsilon)+N(\varepsilon),H\Phi_{M}\right),
\]
and again by the definition of $\Mod$ 
\begin{align*}
&(\Mod,H\Phi_{M})\\
= &-\left(\frac{\lambda_{s}}{\lambda}+b_{1}\right)(\Lambda\tilde{Q}_{b},H\Phi_{M})+[(b_{1})_{s}+b_{1}^{2}(1+c_{b_1})-b_{2}]\left(\tilde{T}_{1}+\chi_{B_{1}}\sum_{j=2}^{4}\frac{\partial S_{j}}{\partial b_{1}},H\Phi_{M}\right)\\
 & +[(b_{2})_{s}+b_{1}b_{2}(3+c_{b_1})]\left(\tilde{T}_{2}+\chi_{B_{1}}\sum_{j=3}^{4}\frac{\partial S_{j}}{\partial b_{2}},H\Phi_{M}\right)\\
= &\  [(b_{1})_{s}+b_{1}^{2}(1+c_{b_1})-b_{2}](-\Lambda Q,\Phi_{M})+[(b_{1})_{s}+b_{1}^{2}(1+c_{b_1})-b_{2}]\left(\sum_{j=2}^{4}\frac{\partial S_{j}}{\partial b_{1}},H\Phi_{M}\right)\\
 & +[(b_{2})_{s}+b_{1}b_{2}(3+c_{b_1})]\left(\sum_{j=3}^{4}\frac{\partial S_{j}}{\partial b_{2}},H\Phi_{M}\right)
\end{align*}
The above computation yields 
\begin{equation}
|(b_{1})_{s}+b_{1}^{2}(1+c_{b_1})-b_{2}|\lesssim b_{1}^{3+\frac{1}{2}}+b_{1}M^{C}|D(t)|.
\end{equation}
Finally, taking the inner product of (\ref{dynamicepsilon}) with $\Phi_{M}$,
we obtain 
\begin{equation}
\left|\frac{\lambda_{s}}{\lambda}+b_{1}\right|\lesssim b_{1}^{3+\frac{1}{2}}+b_{1}M^{C}|D(t)|.\label{step2}
\end{equation}

\noindent \emph{\uline{Step 3: Conclude estimates.}} From (\ref{step1})-(\ref{step2}),
it is not hard to get 
\begin{equation}
|D(t)|\lesssim\frac{1}{\sqrt{\log M}}\left(\sqrt{\varXi_{6}}+\frac{b_{1}^{3}}{|\log b_{1}|}\right)+b_{1}^{3+\frac{1}{2}}.
\end{equation}
Inject this into (\ref{step1})-(\ref{step2}) and we get the desired
results. 
\end{proof}
Unfortunately, (\ref{Modeq2}) is not enough to derive the sharp blow-up rate of $\lambda$, because $b_1b_2c_{b_1}$ (up to $\sqrt{\log M}$) is about the same size with $b_1^3/|\log b_1|$ on the right hand side. To get the sharp blow-up rate and close the bootstrap, we need the following improved
bound for $b_{2}$. 
\begin{prop}\label{prop:improved modulation}
(Improved modulation) Let $\delta>0$ be small enough, $B_{\delta}:=\frac{1}{b_{1}^{\delta}}$
and 
\begin{equation}
\tilde{b}_{2}:=b_{2}+\frac{(H^{2}\varepsilon,\chi_{B_{\delta}}\Lambda Q)}{64\delta|\log b_{1}|},\label{defb2'}
\end{equation}
then 
\begin{equation}
|\tilde{b}_{2}-b_{2}|\lesssim b_{1}^{2+\frac{1}{2}},\label{b2'-b2}
\end{equation}
\begin{equation}
|(\tilde{b}_{2})_{s}+b_{1}\tilde{b}_{2}(3+c_{b_1})|\lesssim\frac{C(M)}{\sqrt{|\log b_{1}|}}\left[\sqrt{\varXi}_{6}+\frac{b_{1}^{3}}{|\log b_{1}|}\right].\label{Modeq2'}
\end{equation}
\end{prop}

\begin{proof}
As above, we replace $H^{2}\Phi_{M}$ by $H^{2}\chi_{B_{\delta}}\Lambda Q$
and take the inner product with (\ref{dynamicepsilon}) 
\begin{align*}
(\Mod,H^{2}\chi_{B_{\delta}}\Lambda Q)= & -\frac{d}{ds}(H^{2}\varepsilon,\chi_{B_{\delta}}\Lambda Q)+(H^{2}\varepsilon,(\partial_{s}\chi_{B_{\delta}})\Lambda Q)-(H^{3}\varepsilon,\chi_{B_{\delta}}\Lambda Q)\\
 & -(H^{2}\tilde{\Psi}_{b},\chi_{B_{\delta}}\Lambda Q)+(\frac{\lambda_{s}}{\lambda}\Lambda\varepsilon+L(\varepsilon)+N(\varepsilon),H^{2}\chi_{B_{\delta}}\Lambda Q).
\end{align*}
Since $B_{\delta}\leq B_{0}\leq B_{1}$ for $\delta$ small enough,
we get 
\begin{align*}
&(\Mod,H^{2}\chi_{B_{\delta}}\Lambda Q)\\
={}&[(b_{1})_{s}+b_{1}^{2}(1+c_{b_1})-b_{2}]\left(\sum_{j=2}^{4}\frac{\partial S_{j}}{\partial b_{1}},H^{2}\chi_{B_{\delta}}\Lambda Q\right)\\
 & +[(b_{2})_{s}+b_{1}b_{2}(3+c_{b_1})]\left(\sum_{j=3}^{4}\frac{\partial S_{j}}{\partial b_{2}},H^{2}\chi_{B_{\delta}}\Lambda Q\right)+[(b_{2})_{s}+b_{1}b_{2}(3+c_{b_1})](\Lambda Q,\chi_{B_{\delta}}\Lambda Q)\\
= {}&\frac{b_{1}}{b_{1}^{C\delta}}O\left(\sqrt{\varXi_{6}}+\frac{b_{1}^{3}}{|\log b_{1}|}\right)+[(b_{2})_{s}+b_{1}b_{2}(3+c_{b_1})](64\delta|\log b_{1}|+O(1))\\
={} &[(b_{2})_{s}+b_{1}b_{2}(3+c_{b_1})]\cdotp64\delta|\log b_{1}|+O\left(\sqrt{\varXi_{6}}+\frac{b_{1}^{3}}{|\log b_{1}|}\right)
\end{align*}
where we used (\ref{Modeq1}) and ( \ref{Modeq2}). We estimate 
\begin{gather*}
|(H^{2}\varepsilon,(\partial_{s}\chi_{B_{\delta}})\Lambda Q)| =|(H^{2}\varepsilon,(b_{1})_{s}(\partial_{b_{1}}\chi_{B_{\delta}})\Lambda Q)| \lesssim\frac{b_{1}}{b_{1}^{C\delta}}\cdotp C(M)\sqrt{\varXi_{6}}\lesssim\sqrt{\varXi_{6}},\\
|(H^{3}\varepsilon,\chi_{B_{\delta}}\Lambda Q)|\lesssim||H^{3}\varepsilon||_{L^{2}}||\chi_{B_{\delta}}\Lambda Q||_{L^{2}}\lesssim\sqrt{|\log b_{1}|}\sqrt{\varXi_{6}},\\
|(H^{2}\tilde{\Psi}_{b},\chi_{B_{\delta}}\Lambda Q)| \lesssim||H^{2}\tilde{\Psi}_{b}||_{L^{2}(y\leq2B_{\delta})}||\chi_{B_{\delta}}\Lambda Q||_{L^{2}} \lesssim\frac{b_{1}^{5}}{b_{1}^{C\delta}}\sqrt{|\log b_{1}|}\lesssim\frac{b_{1}^{3}}{|\log b_{1}|},\\
\left|\left(\frac{\lambda_{s}}{\lambda}\Lambda\varepsilon+L(\varepsilon)+N(\varepsilon),H^{2}\chi_{B_{\delta}}\Lambda Q\right)\right|\lesssim\frac{b_{1}}{b_{1}^{C\delta}}\cdotp C(M)\sqrt{\varXi_{6}}\lesssim\sqrt{\varXi_{6}}.
\end{gather*}
The collection of these bounds yields the preliminary estimate 
\begin{equation}
\begin{split}
&\frac{d}{ds}(H^{2}\varepsilon,\chi_{B_{\delta}}\Lambda Q)+[(b_{2})_{s}+b_{1}b_{2}(3+c_{b_1})]\cdotp64\delta|\log b_{1}|\\
\lesssim{}&C(M)\sqrt{|\log b_{1}|}\left(\sqrt{\varXi_{6}}+\frac{b_{1}^{3}}{|\log b_{1}|}\right).
\end{split}
\end{equation}
We also have 
\begin{equation}
(H^{2}\varepsilon,\chi_{B_{\delta}}\Lambda Q)\lesssim\frac{1}{b_{1}^{C\delta}}\cdotp\sqrt{\varXi_{6}}\lesssim b_{1}^{2+\frac{1}{2}},
\end{equation}
which bounds the deviation 
\begin{equation}
|\tilde{b}_{2}-b_{2}|=\frac{|(H^{2}\varepsilon,\chi_{B_{\delta}}\Lambda Q)|}{64\delta|\log b_{1}|}\lesssim b_{1}^{2+\frac{1}{2}}.
\end{equation}
The proposition follows from 
\begin{align*}
|(\tilde{b}_{2})_{s}+b_{1}\tilde{b}_{2}(3+c_{b_1})| & =\left|\frac{d}{ds}\frac{(H^{2}\varepsilon,\chi_{B_{\delta}}\Lambda Q)}{64\delta|\log b_{1}|}+[(b_{2})_{s}+b_{1}b_{2}(3+c_{b_1})]+O(b_{1}^{3+\frac{1}{2}})\right|\\
 & =\left|\frac{1}{64\delta|\log b_{1}|}\cdotp\frac{d}{ds}(H^{2}\varepsilon,\chi_{B_{\delta}}\Lambda Q)+[(b_{2})_{s}+b_{1}b_{2}(3+c_{b_1})]+O(b_{1}^{3+\frac{1}{2}})\right| \\
 & \lesssim\frac{C(M)}{\sqrt{|\log b_{1}|}}\left(\sqrt{\varXi_{6}}+\frac{b_{1}^{3}}{|\log b_{1}|}\right). 
\end{align*}
This completes the proof.
\end{proof}

\subsection{Lyapounov monotonicity}

We now turn to derive a suitable Lyapounov functional for $\varXi_{6}$
energy. This is crucial to close the bootstrap in Proposition \ref{prop:bootstrap}.
\begin{prop}
\label{prop:Lya6}We have 
\begin{equation}
\frac{d}{dt}\left[\frac{1}{\lambda^{10}}\left(\varXi_{6}+O(b_{1}^{\frac{1}{4}}\cdotp\frac{b_{1}^{6}}{|\log b_{1}|^{2}}\right)\right]\leq C\frac{b_{1}}{\lambda^{12}}\left[\frac{b_{1}^{6}}{|\log b_{1}|^{2}}+\frac{\varXi_{6}}{\sqrt{\log M}}+\frac{b_{1}^{3}}{|\log b_{1}|}\sqrt{\varXi_{6}}\right]\label{Lya6}
\end{equation}
for some constant $C>0$ independent of $M$ and $K$ if $b_{1}$ is
small enough. 
\end{prop}

\begin{proof}
\emph{\uline{Step 1: Suitable derivatives.}} Define $w(t,r)=\frac{1}{\lambda}\varepsilon(s,y)$,  in abbreviation $w(t,\cdot)=(\varepsilon(s,\cdot))_\lambda$. We also denote 
\begin{equation}
w_{2k}=H_{\lambda}^{k}w,\quad \text{where}\ H_{\lambda}:=-\Delta_{r}-\frac{1}{\lambda}V_{\lambda}:=-\Delta_{r}-\bar{V}.
\end{equation}
Then (\ref{dynamicepsilon}) becomes 
\begin{equation}
\partial_{t}w+H_{\lambda}w=\lambda^{-2}\mathscr{F}_{\lambda}.\label{w0eq}
\end{equation}
Using commutator identity
\begin{equation}
[\partial_{t},H_{\lambda}]=\partial_{t}H_{\lambda}-H_{\lambda}\partial_{t}=-\partial_{t}\bar{V},\label{commutator-1}
\end{equation}
we further derive 
\begin{equation}
\partial_{t}w_{2}+H_{\lambda}w_{2}=-\partial_{t}\bar{V}w+H_{\lambda}\left(\lambda^{-2}\mathscr{F}_{\lambda}\right),\label{w2eq}
\end{equation}
\begin{equation}
\partial_{t}w_{4}+H_{\lambda}w_{4}=-\partial_{t}\bar{V}w_{2}-H_{\lambda}(\partial_{t}\bar{V}\cdotp w)+H_{\lambda}^{2}(\lambda^{-2}\mathscr{F}_{\lambda}),\label{w4eq}
\end{equation}
\begin{equation}
\partial_{t}w_{6}+H_{\lambda}w_{6}=-\partial_{t}\bar{V}w_{4}-H_{\lambda}(\partial_{t}\bar{V}\cdotp w_{2})-H_{\lambda}^{2}(\partial_{t}\bar{V}\cdotp w)+H_{\lambda}^{3}(\lambda^{-2}\mathscr{F}_{\lambda}).\label{w6eq}
\end{equation}

\noindent \emph{\uline{Step 2: Energy identity.}}\emph{ }We first
note 
\[
\partial_{t}\bar{V}=\frac{\lambda_{s}}{\lambda}{\hat{V}},\quad \text{with}\ {\hat{V}}=\frac{-1}{\lambda^{3}}(V+\Lambda V)_{\lambda},
\]
then using (\ref{w6eq}) we compute 
\begin{equation}\label{w6^2}
\begin{split}
&\frac{1}{2}\frac{d}{dt}\int w_{6}^{2}\\
={}&\int w_{6}\partial_{t}w_{6}\\
={}& -\int w_{6}\left[H_{\lambda}w_{6}+\partial_{t}\bar{V}w_{4}+H_{\lambda}(\partial_{t}\bar{V}\cdotp w_{2})+H_{\lambda}^{2}(\partial_{t}\bar{V}\cdotp w)-H_{\lambda}^{3}(\lambda^{-2}\mathscr{F}_{\lambda})\right]\\
={}& -\int w_{6}H_{\lambda}w_{6}+\left(\frac{\lambda_{s}}{\lambda}+b_{1}\right)\int w_{6}\left[{\hat{V}}w_{4}+H_{\lambda}({\hat{V}}w_{2})+H_{\lambda}^{2}({\hat{V}}w)\right]\\
 & -\int w_{6}\left[b_{1}{\hat{V}}w_{4}+H_{\lambda}(b_{1}{\hat{V}}w_{2})+H_{\lambda}^{2}(b_{1}{\hat{V}}w)-H_{\lambda}^{3}(\lambda^{-2}\mathscr{F}_{\lambda})\right].
\end{split}
\end{equation}
We further process those terms in the last line of (\ref{w6^2}).
By (\ref{w4eq}), the first term becomes 
\[
-\int b_{1}{\hat{V}}w_{4}w_{6}=\int b_{1}{\hat{V}}w_{4}\left[\partial_{t}w_{4}+\partial_{t}\bar{V}\cdotp w_{2}+H_{\lambda}(\partial_{t}\bar{V}\cdotp w)-H_{\lambda}^{2}(\lambda^{-2}\mathscr{F}_{\lambda})\right]
\]
which produces a boundary term from integration-by-parts 
\[
\int b_{1}{\hat{V}}w_{4}\partial_{t}w_{4}=\frac{1}{2}\frac{d}{dt}\int b_{1}{\hat{V}}w_{4}^{2}-\frac{1}{2}\int\partial_{t}(b_{1}{\hat{V}})w_{4}^{2}.
\]
Similarly the third term becomes 
\[
-\int w_{6}H_{\lambda}^{2}(b_{1}{\hat{V}}w)=\int H_{\lambda}^{2}(b_{1}{\hat{V}}w)\left[\partial_{t}w_{4}+\partial_{t}\bar{V}\cdotp w_{2}+H_{\lambda}(\partial_{t}\bar{V}\cdotp w)-H_{\lambda}^{2}(\lambda^{-2}\mathscr{F}_{\lambda})\right],
\]
and applying (\ref{commutator-1}) after integration-by-parts, we
have 
\begin{align*}
\int H_{\lambda}^{2}(b_{1}{\hat{V}}w)\partial_{t}w_{4}= &\ \frac{d}{dt}\int H_{\lambda}^{2}(b_{1}{\hat{V}}w)w_{4}-\int w_{4}\partial_{t}H_{\lambda}^{2}(b_{1}{\hat{V}}w)\\
= & \ \frac{d}{dt}\int H_{\lambda}^{2}(b_{1}{\hat{V}}w)w_{4}+\int w_{4}(\partial_{t}\bar{V})H_{\lambda}(b_{1}{\hat{V}}w)+\int w_{4}H_{\lambda}[(\partial_{t}\bar{V})b_{1}{\hat{V}}w]\\
 &\ -\int w_{4}H_{\lambda}^{2}(\partial_{t}(b_{1}{\hat{V}})w)-\int w_{4}H_{\lambda}^{2}(b_{1}{\hat{V}}\partial_{t}w).
\end{align*}
Now use (\ref{w0eq}) to replace $\partial_{t}w$ and we find 
\begin{align*}
-\int w_{4}H_{\lambda}^{2}(b_{1}{\hat{V}}\partial_{t}w) & =\int w_{4}H_{\lambda}^{2}(b_{1}{\hat{V}}(w_{2}-\lambda^{-2}\mathscr{F}_{\lambda}))\\
 & =\int w_{6}H_{\lambda}(b_{1}{\hat{V}}w_{2})-\int w_{4}H_{\lambda}^{2}(b_{1}{\hat{V}}\cdotp\lambda^{-2}\mathscr{F}_{\lambda}).
\end{align*}
which cancels the second term in the last line of (\ref{w6^2}).

To sum up, there holds the energy identity 
\begin{equation}\label{energy identity}
\begin{split}
&\frac{1}{2}\frac{d}{dt}\int\left[w_{6}^{2}-b_{1}{\hat{V}}w_{4}^{2}-2w_{4}H_{\lambda}^{2}\left(b_{1}{\hat{V}}w\right)\right]\\
 ={}&   -\int w_{6}H_{\lambda}w_{6}+\left(\frac{\lambda_{s}}{\lambda}+b_{1}\right)\int w_{6}\left[{\hat{V}}w_{4}+H_{\lambda}({\hat{V}}w_{2})+H_{\lambda}^{2}({\hat{V}}w)\right] \\
   & +\int (b_{1}{\hat{V}}) w_{4}(\partial_{t}\bar{V}) w_{2}+\int (b_{1}{\hat{V}})w_{4} H_{\lambda}(\partial_{t}\bar{V}\cdotp  w)-\frac{1}{2}\int\partial_{t}(b_{1}{\hat{V}})\cdotp w_{4}^{2} \\
  & +\int(b_{1}{\hat{V}})w_{4}(\partial_{t}\bar{V})w_{2}+\int(b_{1}{\hat{V}})w_{4}H_{\lambda}(\partial_{t}\bar{V}\cdotp w)-\frac{1}{2}\int\partial_{t}(b_{1}{\hat{V}})\cdotp w_{4}^{2} \\
   & +\int H_{\lambda}^{2}(b_{1}{\hat{V}}w)(\partial_{t}\bar{V})w_{2}+\int H_{\lambda}^{2}(b_{1}{\hat{V}}w)H_{\lambda}(\partial_{t}\bar{V}\cdotp w) \\
   & +\int w_{4}(\partial_{t}\bar{V})H_{\lambda}(b_{1}{\hat{V}}w)+\int w_{4}H_{\lambda}\left(\partial_{t}\bar{V}\cdotp b_{1}{\hat{V}} w\right)-\int w_{4}H_{\lambda}^{2}\left[\partial_{t}(b_{1}{\hat{V}})\cdotp w\right] \\
   & -\int b_{1}{\hat{V}}w_{4}H_{\lambda}^{2}(\lambda^{-2}\mathscr{F}_{\lambda})-\int H_{\lambda}^{2}(b_{1}{\hat{V}}w)H_{\lambda}^{2}(\lambda^{-2}\mathscr{F}_{\lambda})-\int w_{4}H_{\lambda}^{2}(b_{1}{\hat{V}}\lambda^{-2}\mathscr{F}_{\lambda})\\
   &+\int w_{6}H_{\lambda}^{3}(\lambda^{-2}\mathscr{F}_{\lambda}).
   \end{split}
\end{equation}

\noindent Below we estimate (\ref{energy identity}) term by term
to derive (\ref{Lya6}). The estimates use heavily the interpolation
bounds in Appendix B.

\noindent \emph{\uline{Step 3: Lower order quadratic terms.}}

(i) The first term on the RHS of (\ref{energy identity}) is controlled
by \eqref{sub-coerH} and our assumption (\ref{exit3}) of $\tau$
\[
-\int w_{6}H_{\lambda}w_{6}=-\frac{1}{\lambda^{12}}\int\varepsilon_{6}H\varepsilon_{6}\lesssim\frac{1}{\lambda^{12}}(\varepsilon_6,\psi)^2\lesssim \frac{1}{\lambda^{12}}\varsigma^{6}(\varepsilon,\psi)^{2}\lesssim\frac{b_{1}}{\lambda^{12}}\frac{b_{1}^{6}}{|\log b_{1}|^{2}}.
\]


(ii) Next, since 
\[
|\partial_{y}^{i}{\hat{V}}|\lesssim\frac{1}{\lambda^{4}}\frac{1}{1+y^{4+i}},
\]
we estimate the second term using Cauchy-Schwarz and interpolation
bounds in Appendix B
\begin{align}
\begin{split}
&\left|\left(\frac{\lambda_{s}}{\lambda}+b_{1}\right)\int w_{6}\left[{\hat{V}}w_{4}+H_{\lambda}({\hat{V}}w_{2})+H_{\lambda}^{2}({\hat{V}}w)\right]\right|\\
\lesssim&\  b_{1}^{3+\frac{1}{2}}\cdotp\frac{1}{\lambda^{12}}\int\sum_{i=0}^{4}\frac{|\varepsilon_{6}||\partial_{y}^{i}\varepsilon|}{1+y^{8-i}}\lesssim\frac{b_{1}}{\lambda^{12}}\cdotp\frac{b_{1}^{6}}{|\log b_{1}|^{2}}.
\end{split}
\end{align}

(iii) From the definition of $\bar{V},{\hat{V}}$ and modulation
equation (\ref{Modeq1}), we get 
\[
|\partial_{y}^{i}\partial_{t}\bar{V}|\lesssim\frac{b_{1}}{\lambda^{4}}\cdotp\frac{1}{1+y^{4+i}},\quad  |\partial_{y}^{i}\partial_{t}(b_{1}{\hat{V}})|\lesssim\frac{b_{1}^{2}}{\lambda^{6}}\cdotp\frac{1}{1+y^{4+i}},
\]
so the next three lines in (\ref{energy identity}) is bounded by
\begin{align*}
 &\left|\int (b_{1}{\hat{V}}) w_{4}(\partial_{t}\bar{V}) w_{2}+\int (b_{1}{\hat{V}})w_{4} H_{\lambda}(\partial_{t}\bar{V}\cdotp  w)-\frac{1}{2}\int\partial_{t}(b_{1}{\hat{V}})\cdotp w_{4}^{2}\right. \\
  & +\int(b_{1}{\hat{V}})w_{4}(\partial_{t}\bar{V})w_{2}+\int(b_{1}{\hat{V}})w_{4}H_{\lambda}(\partial_{t}\bar{V}\cdotp w)-\frac{1}{2}\int\partial_{t}(b_{1}{\hat{V}})\cdotp w_{4}^{2} \\
   &\left. +\int H_{\lambda}^{2}(b_{1}{\hat{V}}w)(\partial_{t}\bar{V})w_{2}+\int H_{\lambda}^{2}(b_{1}{\hat{V}}w)H_{\lambda}(\partial_{t}\bar{V}\cdotp w) \right|\\
 \lesssim{} & \frac{b_{1}^{2}}{\lambda^{12}}\cdotp\int\sum_{0\leq i,j\leq4}\frac{|\partial_{y}^{i}\varepsilon||\partial_{y}^{j}\varepsilon|}{1+y^{12-i-j}}\lesssim\frac{b_{1}}{\lambda^{12}}\cdotp\frac{b_{1}^{6}}{|\log b_{1}|^{2}}.
\end{align*}

(iv) The boundary terms are estimated similarly
\[
\left|\int-\frac{1}{2}b_{1}{\hat{V}}w_{4}^{2}-w_{4}H_{\lambda}^{2}(b_{1}{\hat{V}}w)\right|\lesssim\frac{b_{1}}{\lambda^{10}}\int\frac{|\varepsilon_{4}|^{2}}{1+y^{4}}+\sum_{i=0}^{4}\frac{|\varepsilon_{4}||\partial_{y}^{i}\varepsilon|}{1+y^{8-i}}\lesssim\frac{b_{1}^{\frac{1}{2}}}{\lambda^{10}}\cdotp\frac{b_{1}^{6}}{|\log b_{1}|^{2}}.
\]

\noindent \emph{\uline{Step 4: Further use of dissipation.}} Finally
we deal with the $\mathscr{F}$ terms. We need to treat the term $\int w_{6}H_{\lambda}^{3}(\lambda^{-2}\mathscr{F}_{\lambda})$
carefully. Let 
\[
\mathscr{F}=\mathscr{F}^{0}+\mathscr{F}^{1},\quad\mathscr{F}^{0}=-\tilde{\Psi}_{b}-\Mod,\quad\mathscr{F}^{1}=L(\varepsilon)+N(\varepsilon),
\]
then 
\begin{align}\label{w6H3-decomp}
\begin{split}
&\int w_{6}H_{\lambda}^{3}(\lambda^{-2}\mathscr{F}_{\lambda})\\
={}&\int w_{6}H_{\lambda}^{3}\left[{\lambda^{-2}}(\mathscr{F}_{\lambda}^{0}+\mathscr{F}_{\lambda}^{1})\right]\\
={}&\int w_{6}H_{\lambda}^{3}(\lambda^{-2}\mathscr{F}_{\lambda}^{0})\\
&-\int H_{\lambda}^{3}(\lambda^{-2}\mathscr{F}_{\lambda}^{1})\left[\partial_{t}w_{4}+\partial_{t}\bar{V}\cdotp w_{2}+H_{\lambda}(\partial_{t}\bar{V}\cdotp w)-H_{\lambda}^{2}\lambda^{-2}\mathscr{F}\right].
 \end{split}
\end{align}
As before, we integrate by parts and use (\ref{commutator-1})
to compute 
\begin{align}\label{w4H3-decomp}
\begin{split}
&-\int\partial_{t}w_{4}H_{\lambda}^{3}(\lambda^{-2}\mathscr{F}_{\lambda}^{1})\\
={} & -\frac{d}{dt}\int w_{4}H_{\lambda}^{3}(\lambda^{-2}\mathscr{F}_{\lambda}^{1})+\int w_{4}\partial_{t}H_{\lambda}^{3}(\lambda^{-2}\mathscr{F}_{\lambda}^{1})\\
={} & -\frac{d}{dt}\int w_{6}H_{\lambda}^{2}\mathscr{F}_{\lambda}^{1}-\int w_{4}(\partial_{t}\bar{V})H_{\lambda}^{2}\mathscr{F}_{\lambda}^{1}-\int w_{4}H_{\lambda}\left[\partial_{t}\bar{V}H_{\lambda}(\lambda^{-2}\mathscr{F}_{\lambda}^{1})\right]\\
 & -\int w_{4}H_{\lambda}^{2}(\partial_{t}\bar{V}\cdotp\lambda^{-2}\mathscr{F}_{\lambda}^{1})+\int w_{4}H_{\lambda}^{3}\left[\partial_{t}(\lambda^{-2}\mathscr{F}_{\lambda}^{1})\right],
 \end{split}
\end{align}
where by direct calculation 
{
\[
\int w_{4}H_{\lambda}^{3}(\partial_{t}(\lambda^{-2}\mathscr{F}_{\lambda}^{1}))=-\int w_{6}H_{\lambda}^{2}\left(\frac{\lambda_{s}}{\lambda^5}(2\mathscr{F}^{1}+\Lambda\mathscr{F}^{1})_{\lambda}\right).
\]}
We now claim the following bounds: 
\begin{gather}
\sum_{i=0}^{4}\int\left(\1_{y\le1}+\frac{|\log y|^{2}}{y^{12-2i}}\1_{y\ge1}\right)(|\partial_{y}^{i}\mathscr{F}^{0}|^{2}+|\partial_{y}^{i}\mathscr{F}^{1}|^{2})+\int|H^{2}\Lambda\mathscr{F}^{1}|^{2} \nonumber\\
\lesssim{}\frac{b_{1}^{6}}{|\log b_{1}|^{2}}+\frac{\varXi_{6}}{\log M},\label{claim1}
\end{gather}
\begin{equation}
\int|H^{3}\mathscr{F}^{0}|^{2}\lesssim b_{1}^{2}\left[\frac{b_{1}^{6}}{|\log b_{1}|^{2}}+\frac{\varXi_{6}}{\log M}\right],\label{claim2}
\end{equation}
\begin{equation}
\int|H^{2}\mathscr{F}^{1}|^{2}\lesssim b_{1}^{\frac{1}{2}}\cdotp\frac{b_{1}^{6}}{|\log b_{1}|^{2}},\label{claim3}
\end{equation}
\begin{equation}
\int|H^{3}\mathscr{F}^{1}|^{2}\leq C_{\gamma}b_{1}^{2-\gamma}\cdotp\frac{b_{1}^{6}}{|\log b_{1}|^{2}},\quad\gamma>0\text{ can be taken arbitrarily small.}\label{claim4}
\end{equation}
With these bounds we estimate terms concerning $\mathscr{F}$. First,
for those in the next-to-last line of (\ref{energy identity}) 
we have 
\begin{align*}
&\left|-\int b_{1}{\hat{V}}w_{4}H_{\lambda}^{2}(\lambda^{-2}\mathscr{F}_{\lambda})-\int H_{\lambda}^{2}(b_{1}{\hat{V}}w)H_{\lambda}^{2}(\lambda^{-2}\mathscr{F}_{\lambda})-\int w_{4}H_{\lambda}^{2}(b_{1}{\hat{V}}\lambda^{-2}\mathscr{F}_{\lambda})\right|\\
\lesssim{}&\frac{b_{1}}{\lambda^{12}}\int\sum_{0\leq i,j\leq4}\frac{|\partial_{y}^{i}\varepsilon||\partial_{y}^{j}\mathscr{F}|}{1+y^{12-i-j}}\lesssim\frac{b_{1}}{\lambda^{12}}\left[\frac{\varXi_{6}}{\sqrt{\log M}}+\frac{b_{1}^{3}}{|\log b_{1}|}\sqrt{\varXi_{6}}\right].
\end{align*}
To estimate $\int w_{6}H_{\lambda}^{3}(\lambda^{-2}\mathscr{F}_{\lambda})$, we use \eqref{w6H3-decomp} and \eqref{w4H3-decomp}. From (\ref{claim2}) we derive 
\[
\left|\int w_{6}H_{\lambda}^{3}(\lambda^{-2}\mathscr{F}_{\lambda}^{0})\right|\lesssim\frac{1}{\lambda^{12}}\int|\varepsilon_{6}||H^{3}\mathscr{F}^{0}|\lesssim\frac{b_{1}}{\lambda^{12}}\left(\frac{\varXi_{6}}{\sqrt{\log M}}+\frac{b_{1}^{3}}{|\log b_{1}|}\sqrt{\varXi_{6}}\right),
\]
and from (\ref{claim1}) and  (\ref{claim3})
\begin{align*}
 & \left|-\int w_{4}(\partial_{t}\bar{V})H_{\lambda}^{2}\mathscr{F}_{\lambda}^{1}-\int w_{4}H_{\lambda}(\partial_{t}\bar{V}H_{\lambda}(\lambda^{-2}\mathscr{F}_{\lambda}^{1}))-\int w_{4}H_{\lambda}^{2}(\partial_{t}\bar{V}\cdotp\lambda^{-2}\mathscr{F}_{\lambda}^{1})\right.\\
 &\left. -\int w_{6}H_{\lambda}^{2}\left(\frac{\lambda_{s}}{\lambda^5}(2\mathscr{F}^{1}+\Lambda\mathscr{F}^{1})_{\lambda}\right)\right|\\
\lesssim{} &\frac{b_{1}}{\lambda^{12}}\int\sum_{i=0}^{4}\frac{|\varepsilon_{4}||\partial_{y}^{i}\mathscr{F}^{1}|}{1+y^{8-i}}+|\varepsilon_{6}||H^{2}\mathscr{F}^{1}+H^{2}\Lambda\mathscr{F}^{1}|\\
\lesssim{} &\frac{b_{1}}{\lambda^{12}}\left(\frac{\varXi_{6}}{\sqrt{\log M}}+\frac{b_{1}^{3}}{|\log b_{1}|}\sqrt{\varXi_{6}}\right).
\end{align*}
Picking $\gamma=1/2$ in (\ref{claim4}), other non-boundary terms
are controlled by 
\begin{align*}
 & \left|-\int H_{\lambda}^{3}(\lambda^{-2}\mathscr{F}_{\lambda}^{1})(\partial_{t}\bar{V}\cdotp w_{2}+H_{\lambda}(\partial_{t}\bar{V}\cdotp w)-H_{\lambda}^{2}(\lambda^{-2}\mathscr{F}_{\lambda}))\right|\\
={} & \left|-\int H_{\lambda}^{2}(\lambda^{-2}\mathscr{F}_{\lambda}^{1})(H_{\lambda}(\partial_{t}\bar{V}\cdotp w_{2})+H_{\lambda}^{2}(\partial_{t}\bar{V}\cdotp w)-H_{\lambda}^{3}(\lambda^{-2}\mathscr{F}_{\lambda}))\right|\\
\lesssim{} & \frac{1}{\lambda^{12}}\int|H^{2}\mathscr{F}^{1}|\left(b_{1}\sum_{i=0}^{4}\frac{|\partial_{y}^{i}\varepsilon|}{1+y^{8-i}}+|H^{3}\mathscr{F}|\right)
\lesssim  \frac{b_{1}}{\lambda^{12}}\cdotp\frac{b_{1}^{6}}{|\log b_{1}|^{2}}.
\end{align*}
We can also estimate the new boundary term as
\[
\left|\int w_{6}H_{\lambda}^{2}(\lambda^{-2}\mathscr{F}_{\lambda}^{1})\right|\lesssim\frac{1}{\lambda^{10}}\int|\varepsilon_{6}||H^{2}\mathscr{F}^{1}|\lesssim\frac{b_{1}^{\frac{1}{4}}}{\lambda^{10}}\cdotp\frac{b_{1}^{6}}{|\log b_{1}|^{2}}.
\]
The proposition follows from these estimates. Now we turn to the proof
of the claim.

\noindent \emph{\uline{Step 5: $\tilde{\Psi}_{b}$ terms.}}
The contribution of $\tilde{\Psi}_{b}$
in (\ref{claim1}), (\ref{claim2}) is estimated in (\ref{Psi2}) and (\ref{Psi3}).

\noindent \emph{\uline{Step 6: Mod terms.}} Recall (\ref{defMod}) that 
\begin{align*}
\Mod=&\ -\left(\frac{\lambda_{s}}{\lambda}+b_{1}\right)\Lambda\tilde{Q}_{b}+[(b_{1})_{s}+b_{1}^{2}(1+c_{b_1})-b_{2}]\left[\tilde{T}_{1}+\chi_{B_{1}}\sum_{j=2}^{4}\frac{\partial S_{j}}{\partial b_{1}}\right]\\
&\ +[(b_{2})_{s}+b_{1}b_{2}(3+c_{b_1})]\left[\tilde{T}_{2}+\chi_{B_{1}}\sum_{j=3}^{4}\frac{\partial S_{j}}{\partial b_{2}}\right].
\end{align*}

\noindent \emph{Proof of (\ref{claim2}) for $\Mod$}. 
Thanks to modulation
equations (\ref{Modeq1}), (\ref{Modeq2}), it suffices to show 
\begin{equation}
\int|H^{3}\Lambda\tilde{Q}_{b}|^{2}+\int\left|H^{3}(\tilde{T}_{1}+\chi_{B_{1}}\sum_{j=2}^{4}\frac{\partial S_{j}}{\partial b_{1}})\right|^{2}\lesssim b_{1}^{2}|\log b_{1}|^{C},\label{M1}
\end{equation}
\begin{equation}
\int\left|H^{3}(\tilde{T}_{2}+\chi_{B_{1}}\sum_{j=3}^{4}\frac{\partial S_{j}}{\partial b_{2}})\right|^{2}\lesssim b_{1}^{2}.\label{M2}
\end{equation}
For (\ref{M1}), since $H^{3}\Lambda Q=H^{3}T_{1}=0$, we have
\begin{align*}
&\int|H^{3}\Lambda\tilde{Q_{b}}|^{2}+\int|H^{3}(\tilde{T}_{1}+\chi_{B_{1}}\sum_{j=2}^{4}\frac{\partial S_{j}}{\partial b_{1}})|^{2} \\
\lesssim{}&\ \int_{B_{1}\le y\le2B_{1}}|H^{3}\tilde{T}_{1}|^{2}+b_{1}^{2}|\log b_{1}|^{C}\int_{y\leq2B_{1}}\frac{1}{1+y^{12}}\\
  \lesssim{}&\ \int_{B_{1}\leq y\leq2B_{1}}\frac{|\log y|^{2}}{y^{12}}+b_{1}^{2}|\log b_{1}|^{C} \lesssim b_{1}^{2}|\log b_{1}|^{C}.
\end{align*}
To prove (\ref{M2}), we use $H^{3}T_{2}=0$ to estimate 
\begin{equation}
\int|H^{3}\tilde{T}_{2}|^{2}\lesssim\int_{B_{1}\leq y\leq2B_{1}}\frac{|\log y|^{2}}{y^{8}}\lesssim b_{1}^{2}.\label{M2-1}
\end{equation}
Note that (\ref{S4}) implies that 
\begin{equation}
\int\left|H^{3}(\chi_{B_{1}}\frac{\partial S_{4}}{\partial b_{2}})\right|^{2}\lesssim\frac{b_{1}^{2}}{|\log b_{1}|^{2}}\int_{y\leq2B_{1}}\frac{1}{1+y^{4}}\lesssim b_{1}^{2}.\label{M2-2}
\end{equation}
Finally we deal with the term $H^{3}\partial_{b_{2}}S_{3}$. From
(\ref{defS3}), (\ref{defS2}), (\ref{HTheta3}) 
\begin{align*}
H^{3}\partial_{b_{2}}S_{3} & =O(b_{1})(H\Theta_{2}+b_{1}\partial_{b_{1}}H\Theta_{2})+O(b_{1})\left(\1_{y\leq1}+\frac{|\log y|^{C}}{y^{4}}\1_{y\geq1}\right)
\end{align*}
where by the definition of $\Theta_{2}$ 
\[
H\Theta_{2}=H(\Lambda T_{1}-T_{1}+\varSigma_{b_1})=-\Lambda^{2}Q-\Lambda Q+H_{b}+O\left(\frac{|\log y|^{C}}{y^{4}}\1_{y\geq1}\right).
\]
Now we estimate from (\ref{Q}) and (\ref{defSigmab}) 
\[
|-\Lambda^{2}Q-\Lambda Q|\lesssim\frac{1}{1+y^{4}},\quad  |H\varSigma_{b_1}|+|b_{1}\partial_{b_{1}}H\varSigma_{b_1}|\lesssim\frac{1}{|\log b_{1}|(1+y^{2})}\1_{y\leq6B_{0}}
\]
and hence obtain 
\[
|H^{3}\partial_{b_{2}}S_{3}|\lesssim b_{1}\left(\1_{y\leq1}+\frac{1}{|\log b_{1}|y^{2}}\1_{1\le y\leq6B_{0}}+\frac{|\log b_{1}|^{C}}{y^{4}}\1_{y\geq6B_{0}}\right).
\]
Together with (\ref{S3}) we conclude 
\begin{equation}
\int\left|H^{3}(\chi_{B_{1}}\frac{\partial S_{3}}{\partial b_{2}})\right|^{2}\lesssim\int_{y\leq B_{1}}\left|H^{3}\frac{\partial S_{3}}{\partial b_{2}}\right|^{2}+\frac{1}{|\log b_{1}|^{2}}\int_{B_{1}\leq y\leq2B_{1}}\frac{1}{y^{8}}\lesssim b_{1}^{2}.\label{M2-3}
\end{equation}
Combining (\ref{M2-1}), (\ref{M2-2}), and (\ref{M2-3}), we derive (\ref{M2})
and finish the proof of (\ref{claim2}).

\noindent \emph{Proof of (\ref{claim1}) for $\Mod$}. Using rough
bounds for $S_{j}$ we simply estimate for $0\le i\le4$ 
\begin{align}
\begin{split}
&\int\frac{1+|\log y|^{2}}{1+y^{12-2i}}\left[|\partial_{y}^{i}\Lambda\tilde{Q}_{b}|^{2}+\left|\partial_{y}^{i}(\tilde{T}_{1}+\chi_{B_{1}}\sum_{j=2}^{4}\frac{\partial S_{j}}{\partial b_{1}})\right|^{2}+\left|\partial_{y}^{i}(\tilde{T}_{2}+\chi_{B_{1}}\sum_{j=3}^{4}\frac{\partial S_{j}}{\partial b_{2}})\right|^{2}\right]\\
\hspace{-1em}\lesssim{}&\int\frac{1+|\log y|^{C}}{1+y^{8}}\lesssim1
\end{split}
\end{align}
which combined with modulation equations implies the desired result:
\begin{equation}
\sum_{i=0}^{4}\int\frac{1+|\log y|^{2}}{1+y^{12-2i}}|\partial_{y}^{i}\Mod|^{2}\lesssim\frac{b_{1}^{6}}{|\log b_{1}|^{2}}+\frac{\varXi_{6}}{\log M}.
\end{equation}

\noindent \emph{\uline{Step 7: $L(\varepsilon)$ terms.}} Recall that $L(\varepsilon)=3(\tilde{Q}_{b}-Q^{2})\varepsilon$.
We have 
\begin{equation}
|\partial_{y}^{i}(\tilde{Q}_{b}^{2}-Q^{2})|\lesssim b_{1}\left(\1_{y\le1}+\frac{|\log y|^{C}}{y^{2+i}}\1_{1\leq y\leq2B_{1}}\right).
\end{equation}
Using Leibniz rule we estimate 
\begin{equation}
\begin{split}
 & \sum_{i=0}^{4}\int\left(\1_{y\le1}+\frac{|\log y|^{2}}{y^{12-2i}}\1_{y\ge1}\right)|\partial_{y}^{i}L(\varepsilon)|^{2}+\int|H^{2}\Lambda L(\varepsilon)|^{2}\\
\lesssim{} &\  b_{1}^{2}\sum_{i=0}^{5}\left(\int_{y\leq1}|\partial_{y}^{i}\varepsilon|^{2}+\int_{1\leq y\leq2B_{1}}\frac{|\log y|^{C}}{y^{12-2i}}|\partial_{y}^{i}\varepsilon|^{2}\right)\\
\lesssim{} &\  b_{1}^{2}(1+|\log b_{1}|^{C})\varXi_{6}\lesssim\frac{b_{1}^{6}}{|\log b_{1}|^{2}},
\end{split}
\end{equation}
\begin{equation}
\begin{split}
\int|H^{2}L(\varepsilon)|^{2} & \lesssim b_{1}^{2}\sum_{i=0}^{4}\left(\int_{y\leq1}|\partial_{y}^{i}\varepsilon|^{2}+\int_{1\leq y\leq2B_{1}}\frac{|\log y|^{C}}{y^{12-2i}}|\partial_{y}^{i}\varepsilon|^{2}\right)\\
 & \lesssim b_{1}^{2}(1+|\log b_{1}|^{C})\varXi_{6}\lesssim b_{1}\cdotp\frac{b_{1}^{6}}{|\log b_{1}|^{2}},
\end{split}
\end{equation}
\begin{equation}
\begin{split}
\intop|H^{3}L(\varepsilon)|^{2} & \lesssim b_{1}^{2}\sum_{i=0}^{6}\left(\int_{y\le1}|\partial_{y}^{i}\varepsilon|^{2}+\int_{1\leq y\leq2B_{1}}\frac{1+|\log y|^{C}}{y^{16-2i}}|\partial_{y}^{i}\varepsilon|^{2}\right)\label{H3L}\\
 & \lesssim b_{1}^{2}\varXi_{6}\le C_{\gamma}b_{1}^{2-\gamma}\cdotp\frac{b_{1}^{6}}{|\log b_{1}|^{2}},
\end{split}
\end{equation}
which bound the $L(\varepsilon)$ terms in (\ref{claim1}), (\ref{claim3})
and (\ref{claim4}) respectively.

\noindent \emph{\uline{Step 8: $N(\varepsilon)$ terms.}} Recall that $N(\varepsilon)=3\tilde{Q}_{b}\cdotp\varepsilon^{2}+\varepsilon^{3}$.
We split the integral into two parts:

\noindent \emph{Control for $y\leq1$.} We know $|\partial_{y}^{i}\tilde{Q}_{b}|\lesssim1$
and thus 
\begin{align}
\begin{split}
\int_{y\leq1}|H^{2}(\tilde{Q}_{b}\cdotp\varepsilon^{2})|^{2} & \lesssim\sum_{i+j\leq4}\int_{y\leq1}|\partial_{y}^{i}\varepsilon|^{2}|\partial_{y}^{j}\varepsilon|^{2}\\& \lesssim\sum_{i+j\leq4}\int_{y\leq1}|\partial_{y}^{i}\varepsilon|^{2}\cdotp\Vert\partial_{y}^{j}\varepsilon\Vert_{L^{\infty}(y\leq1)}^{2}\\ &\lesssim\varXi_{6}^{2}\lesssim b_{1}\cdotp\frac{b_{1}^{6}}{|\log b_{1}|^{2}},
 \end{split}
\end{align}
where we assume $j\leq i$ (hence $j\le2$) and use (\ref{pointwise2}).
In the following we assume $i\geq j\geq k$ and derive 
\begin{align}
\begin{split}
\int_{y\le1}|H^{2}\varepsilon^{3}|^{2} & \lesssim\sum_{i+j+k\le4}\int_{y\leq1}|\partial_{y}^{i}\varepsilon|^{2}|\partial_{y}^{j}\varepsilon|^{2}|\partial_{y}^{k}\varepsilon|^{2}\\
 & \lesssim\sum_{i+j+k\le4}\int_{y\leq1}|\partial_{y}^{i}\varepsilon|^{2}\cdotp\|\partial_{y}^{j}\varepsilon\|_{L^{\infty}(y\le1)}^{2}\|\partial_{y}^{k}\varepsilon\|_{L^{\infty}(y\le1)}^{2}\\ & \lesssim\varXi_{6}^{3}\lesssim b_{1}\cdotp\frac{b_{1}^{6}}{|\log b_{1}|^{2}}.
 \end{split}
\end{align}

\noindent Similar calculations imply 
\begin{equation}
\sum_{i=0}^{4}\int_{y\leq1}|\partial_{y}^{i}N(\varepsilon)|^{2}+\int_{y\leq1}|H^{2}\Lambda N(\varepsilon)|^{2}\lesssim\frac{b_{1}^{6}}{|\log b_{1}|^{2}},\ 
\end{equation}
\begin{equation}
\int_{y\le1}|H^{3}N(\varepsilon)|^{2}\lesssim b_{1}^{2}\cdotp\frac{b_{1}^{6}}{|\log b_{1}|^{2}}.\label{H3N1}
\end{equation}

\noindent These bounds verify the $y\le1$ part of (\ref{claim1}),
(\ref{claim3}) and (\ref{claim4}).

\noindent \emph{Control for $y\ge1$}. Now we have 
\begin{equation}
|\partial_{y}^{i}\tilde{Q}_{b}|\lesssim\frac{1}{y^{2+i}}+b_{1}\frac{1+|\log y|^{C}}{y^{i}}\cdotp\1_{y\leq2B_{1}}\lesssim\frac{1+|\log              y|^{C}}{y^{2+i}}.
\end{equation}
\emph{Proof of (\ref{claim3}}) for $N(\varepsilon)$, $y\ge1.$ First
for $H^{2}(\tilde{Q}_{b}\cdotp\varepsilon^{2})$ the bound above yields
\begin{align}
\int_{y\geq1}|H^{2}(\tilde{Q}_{b}\cdotp\varepsilon^{2})|^{2} & \lesssim\sum_{i+j\leq4}\int_{y\geq1}\frac{1+|\log y|^{C}}{y^{12-2i-2j}}|\partial_{y}^{i}\varepsilon|^{2}|\partial_{y}^{j}\varepsilon|^{2}\label{H2Qepsilon2}\\
 & \lesssim\sum_{i+j\leq4}\int_{y\geq1}\frac{|\partial_{y}^{i}\varepsilon|^{2}}{y^{8-2i}(1+|\log y|^{2})}\cdotp||y^{j-2}\partial_{y}^{j}\varepsilon(1+|\log y|^{C})||_{L^{\infty}(y\geq1)}^{2}\nonumber \\
 & \lesssim|\log b_{1}|^{C}\varXi_{4}^{2}\lesssim b_{1}\cdotp\frac{b_{1}^{6}}{|\log b_{1}|^{2}}.\nonumber 
\end{align}
Next we treat the term $H^{2}(\varepsilon^{3})$. Since $|\partial_{y}^{i}V|\lesssim\frac{1}{1+y^{4+i}}$,
we get
\begin{equation}
|H^{2}(\varepsilon^{3})|^{2}\lesssim\sum_{i+j+k\le4}\frac{|\partial_{y}^{i}\varepsilon|^{2}|\partial_{y}^{j}\varepsilon|^{2}|\partial_{y}^{k}\varepsilon|^{2}}{y^{8-2i-2j-2k}}.\label{H2epsilon3-0}
\end{equation}
Again assume $i\geq j\geq k$. If $(i,j,k)=(2,2,0)$ we estimate 
\begin{equation}
\int_{y\geq1}|\partial_{y}^{2}\varepsilon|^{2}|\partial_{y}^{2}\varepsilon|^{2}|\varepsilon|^{2}\lesssim\int_{y\geq1}|\partial_{y}^{2}\varepsilon|^{2}\cdotp||\partial_{y}^{2}\varepsilon||_{L^{\infty}(y\geq1)}^{2}\cdotp||\varepsilon||_{L^{\infty}(y\geq1)}^{2},\label{H2epsilon3-1}
\end{equation}
otherwise $k\le j\le1$ and 
\begin{align}\label{H2epsilon3-2}
\begin{split}
&\int_{y\geq1}\frac{|\partial_{y}^{i}\varepsilon|^{2}|\partial_{y}^{j}\varepsilon|^{2}|\partial_{y}^{k}\varepsilon|^{2}}{y^{8-2i-2j-2k}}\\
\lesssim{}&\int_{y\geq1}\frac{|\partial_{y}^{i}\varepsilon|^{2}}{y^{8-2i}(1+|\log y|^{2})}\cdotp||y^{j}\partial_{y}^{j}\varepsilon||_{L^{\infty}(y\geq1)}^{2}\cdotp||y^{k}\partial_{y}^{k}\varepsilon(1+|\log y|^{C})||_{L^{\infty}(y\geq1)}^{2}.
\end{split}
\end{align}
Hence we conclude with the help of Appendix B that 
\begin{equation}
\int_{y\ge1}|H^{2}(\varepsilon^{3})|^{2}\lesssim|\log b_{1}|^{C}\varXi_{2}^{2}\varXi_{4}\lesssim b_{1}^{\frac{1}{2}}\cdotp\frac{b_{1}^{6}}{|\log b_{1}|^{2}}
\end{equation}
which is where the exponent $1/2$ in (\ref{claim3}) comes from.
This completes the proof of (\ref{claim3}).

\noindent \emph{Proof of (\ref{claim1}}) for $N(\varepsilon),y\ge1.$
Note that the term involving $\partial_{y}^{5}\varepsilon$ in $\int_{y\geq1}|H^{2}\Lambda(\tilde{Q}_{b}\cdotp\varepsilon^{2})|^{2}$
is bounded by 
\begin{equation*}
\int_{y\geq1}\frac{1+|\log y|^{C}}{y^{2}}|\partial_{y}^{5}\varepsilon|^{2}|\varepsilon|^{2}\lesssim\int_{y\geq1}\frac{|\partial_{y}^{5}\varepsilon|^{2}}{y^{2}(1+|\log y|^{2})}\cdotp||\varepsilon(1+|\log y|^{C})||_{L^{\infty}(y\geq1)}^{2}\lesssim\frac{b_{1}^{6}}{|\log b_{1}|^{2}}
\end{equation*}
and in $\int_{y\ge1}|H^{2}\Lambda(\varepsilon^{3})|^{2}$ by 
\begin{equation*}
\int_{y\geq1}|\partial_{y}^{5}\varepsilon|^{2}|\varepsilon|^{4}\lesssim\int_{y\geq1}\frac{|\partial_{y}^{5}\varepsilon|^{2}}{y^{2}(1+|\log y|^{2})}\cdotp||\varepsilon(1+|\log y|^{C})||_{L^{\infty}(y\geq1)}^{2}\cdotp||y\varepsilon||_{L^{\infty}(y\geq1)}^{2}\lesssim\frac{b_{1}^{6}}{|\log b_{1}|^{2}},
\end{equation*}
where we use (\ref{Hardy}) and smallness of energy (\ref{exit1}).
Other lower order terms are estimated as before in (\ref{H2Qepsilon2}), (\ref{H2epsilon3-1}) and (\ref{H2epsilon3-2})
.

\noindent \emph{Proof of (\ref{claim4}}) for $N(\varepsilon)$, $y\ge1$.
For $H^{3}(\tilde{Q}_{b}\cdotp\varepsilon^{2})$ we expand 
\begin{equation}
|H^{3}(\tilde{Q}_{b}\cdotp\varepsilon^{2})|^{2}\lesssim\sum_{i+j\leq6}\frac{1+|\log y|^{C}}{y^{16-2i-2j}}|\partial_{y}^{i}\varepsilon|^{2}|\partial_{y}^{j}\varepsilon|^{2},
\end{equation}
assume $i\ge j$, and then estimate 
\begin{align*}
&\int_{y\geq1}\frac{1+|\log y|^{C}}{y^{16-2i-2j}}|\partial_{y}^{i}\varepsilon|^{2}|\partial_{y}^{j}\varepsilon|^{2}\\
\lesssim{}&\begin{dcases}
\int_{y\geq1}\frac{1+|\log y|^{C}}{y^{6}}|\partial_{y}^{3}\varepsilon|^{2}\cdotp\left\|\frac{y\partial_{y}^{3}\varepsilon}{1+\log y}\right\|_{L^{\infty}(y\geq1)}^{2}, & i=j=3,\\
\int_{y\geq1}\frac{|\partial_{y}^{i}\varepsilon|^{2}}{y^{12-2i}(1+|\log y|^{2})}\cdotp||y^{j-2}\partial_{y}^{j}\varepsilon(1+|\log y|^{C})||_{L^{\infty}(y\geq1)}^{2}, & \text{otherwise}.
\end{dcases}
\end{align*}
In both cases, we have 
\begin{equation}
    \int_{y\geq1}\frac{1+|\log y|^{C}}{y^{16-2i-2j}}|\partial_{y}^{i}\varepsilon|^{2}|\partial_{y}^{j}\varepsilon|^{2}\lesssim |\log b_1|^C \varXi_{4}\varXi_{6}\lesssim  b_1^2\frac{b_{1}^{6}}{|\log b_{1}|^{2}}.
\end{equation}
Besides, we need to control the term $H^{3}(\varepsilon^{3})$. First
as in (\ref{H2epsilon3-0}) 
\begin{equation}
H^{3}(\varepsilon^{3})\lesssim\sum_{i+j+k\le6}\frac{|\partial_{y}^{i}\varepsilon|^{2}|\partial_{y}^{j}\varepsilon|^{2}|\partial_{y}^{k}\varepsilon|^{2}}{y^{12-2i-2j-2k}}.
\end{equation}
Assuming $i\ge j\ge k$, we divide the summand into several cases
and summarize the estimates as follows: 
\begin{align*}
    &\int_{y\geq1}\frac{|\partial_{y}^{i}\varepsilon|^{2}|\partial_{y}^{j}\varepsilon|^{2}|\partial_{y}^{k}\varepsilon|^{2}}{y^{12-2i-2j-2k}}\lesssim\\
    &\begin{cases}
\int_{y\geq1}\frac{|\partial_{y}^{i}\varepsilon|^{2}}{y^{12-2i}(1+|\log y|^{2})}\cdotp||y^{j}\partial_{y}^{j}\varepsilon||_{L^{\infty}(y\geq1)}^{2}\cdotp||y^{k}\partial_{y}^{k}\varepsilon(1+|\log y|^{C})||_{L^{\infty}(y\geq1)}^{2}, &j\le1,\\[1em]
\int_{y\geq1}\frac{|\partial_{y}^{i}\varepsilon|}{y^{8-2i}(1+|\log y|^{2})}\cdotp||\partial_{y}^{2}\varepsilon||_{L^{\infty}(y\geq1)}^{2}\cdotp||y^{k}\partial_{y}^{k}\varepsilon(1+|\log y|^{C})||_{L^{\infty}(y\geq1)}^{2}, &j=2,k\le1,\\[1em]
\int_{y\geq1}|\partial_{y}^{2}\varepsilon|^{2}\cdotp||\partial_{y}^{2}\varepsilon||_{L^{\infty}(y\geq1)}^{2}\cdotp||\partial_{y}^{2}\varepsilon||_{L^{\infty}(y\geq1)}^{2}, &(i,j,k)=(2,2,2),\\[1em]
\int_{y\geq1}\frac{|\partial_{y}^{3}\varepsilon|^{2}}{y^{2}(1+|\log y|^{2})}\cdotp||\frac{y\partial_{y}^{3}\varepsilon}{1+\log y}||_{L^{\infty}(y\geq1)}^{2}\cdotp||\varepsilon(1+|\log y|^{2})||_{L^{\infty}(y\geq1)}^{2}, &(i,j,k)=(3,3,0).
\end{cases}
\end{align*}
Using interpolation bounds in Appendix B, we conclude 
\begin{equation}
\int_{y\ge1}|H^{3}N(\varepsilon)|^{2}\lesssim b_{1}^{2}\frac{b_{1}^{6}}{|\log b_{1}|^{2}}\label{H3N2}
\end{equation}
and hence prove (\ref{claim4}). As a result, we finish the whole
proof of Proposition \ref{prop:Lya6}.
\end{proof}
Similarly we have Lyapounov monotonicity for $\varXi_{4}$ and $\varXi_{2}$: 
\begin{prop}
For some constant $C$ independent of $M$ and $K$, there holds  
\begin{equation}
\frac{d}{dt}\left(\frac{1}{\lambda^{6}}\varXi_{4}\right)\leq C\frac{b_{1}^{5}}{\lambda^{8}},\label{Lya4}
\end{equation}
\begin{equation}
\frac{d}{dt}\left(\frac{1}{\lambda^{2}}\varXi_{2}\right)\leq C\frac{b_{1}^{\frac{7}{3}}}{\lambda^{4}}.\label{Lya2}
\end{equation}
\end{prop}

\begin{proof}
To prove (\ref{Lya4}), we compute 
\begin{equation}
\frac{1}{2}\frac{d}{dt}\int w_{4}^{2}=\int w_{4}(-H_{\lambda}w_{4}-\partial_{t}\bar{V}w_{2}-H_{\lambda}(\partial_{t}\bar{V}w)+H_{\lambda}^{2}({\lambda}^{-2}F_{\lambda})).
\end{equation}
Unlike in the proof of (\ref{Lya6}), here we could estimate these
terms directly. The first term is estimated from (\ref{exit3}) 
\begin{equation}
-\int w_{4}H_{\lambda}w_{4}=-\frac{1}{\lambda^{8}}\int\varepsilon_{4}H\varepsilon_{4}\lesssim\frac{1}{\lambda^{8}}(\varepsilon,\psi)^{2}\lesssim\frac{b_{1}^{5}}{\lambda^{8}},
\end{equation}
and the next two follow from $|\partial_{t}\bar{V}|\lesssim\frac{b_{1}}{\lambda^{4}}\cdotp\frac{1}{1+y^{4}}$
\begin{equation}
\left|-\int w_{4}\partial_{t}\bar{V}w_{2}-\int w_{4}H_{\lambda}(\partial_{t}\bar{V}w)\right| \lesssim\frac{b_{1}}{\lambda^{8}}\int\frac{|\varepsilon_{2}||\varepsilon_{4}|+|\varepsilon||\varepsilon_{6}|}{1+y^{4}}\lesssim\frac{b_{1}^{5}}{\lambda^{8}}.
\end{equation}
The last term is controlled thanks to (\ref{claim3}) 
\begin{equation}
\left|\int w_{4}H_{\lambda}^{2}(\lambda^{-2}\mathscr{F}_{\lambda})\right|\lesssim\frac{1}{\lambda^{8}}\int|\varepsilon_{4}||H^{2}\mathscr{F}|\lesssim\frac{b_{1}^{5}}{\lambda^{8}}.
\end{equation}
hence (\ref{Lya4}) is proved.

Now we turn to the proof of (\ref{Lya2}). Similarly 
\begin{equation}
\frac{1}{2}\frac{d}{dt}\int w_{2}^{2}=\int w_{2}[-H_{\lambda}w_{2}-\partial_{t}\bar{V}w+H_{\lambda}({\lambda}^{-2}F_{\lambda})].
\end{equation}
and it is not hard to show 
\begin{equation}
-\int w_{2}H_{\lambda}w_{2}\lesssim\frac{b_{1}^{\frac{7}{3}}}{\lambda^{4}},\quad  \text{and}\quad  \left|-\int w_{2}\partial_{t}\bar{V}w\right|\lesssim\frac{b_{1}^{\frac{7}{3}}}{\lambda^{4}}.
\end{equation}
To estimate the last term we need to bound $\int|H\mathscr{F}|^{2}$.
Recall that $\mathscr{F}=-\tilde{\Psi}-\Mod-L(\varepsilon)+N(\varepsilon)$.
By (\ref{Psi1'}) we have 
\[
\int|H\tilde{\Psi}|^{2}\lesssim b_{1}^{4}|\log b_{1}|^{C}\lesssim b_{1}^{\frac{11}{3}}.
\]
For the $\Mod$ term, since 
\[
|H\Lambda\tilde{Q}_{b}|^{2}+\left|H\left(\tilde{T}_{1}+\chi_{B_{1}}\sum_{j=2}^{4}\frac{\partial S_{j}}{\partial b_{1}}\right)\right|^{2}+\left|H\left(\tilde{T}_{2}+\chi_{B_{1}}\sum_{j=3}^{4}\frac{\partial S_{j}}{\partial b_{2}}\right)\right|^{2}\lesssim|\log b_{1}|^{C}\1_{y\le2B_{1}},
\]
we could derive
\[
\int|H(\Mod)|^{2}\lesssim b_{1}^{4}|\log b_{1}|^{C}\lesssim b_{1}^{\frac{11}{3}}.
\]
For the error term, we directly estimate them, i.e. for $L(\varepsilon)$
we have 
\[
\int|HL(\varepsilon)|^{2}\lesssim b_{1}^{2}\left(\sum_{i=0}^{2}\int_{y\leq1}|\partial_{y}^{i}\varepsilon|^{2}+\sum_{i=0}^{2}\int_{y\geq1}\frac{1+|\log y|^{C}}{y^{8-2i}}|\partial_{y}^{i}\varepsilon|^{2}\right)\lesssim b_{1}^{\frac{11}{3}}.
\]
For $N(\varepsilon)=3\tilde{Q}_{b}\cdotp\varepsilon^{2}+\varepsilon^{3}$ we have
\[\begin{split}
\int|H(\tilde{Q}_{b}\,\varepsilon^{2})|^{2}
&\lesssim \sum_{i+j\leq2}\left(||\partial_{y}^{i}\varepsilon||_{L^{\infty}(y\le1)}^{2}\cdotp||\partial_{y}^{j}\varepsilon||_{L^{\infty}(y\le1)}^{2}
+\int_{y\geq1}\frac{|\partial_{y}^{i}\varepsilon|^{2}}{y^{8-2i}}\cdotp||y^{j}\partial_{y}^{j}\varepsilon||_{L^{\infty}(y\ge1)}^{2}\cdotp|\log b_{1}|^{C}\right)\\
&\lesssim b_{1}^{6}+\varXi_{4}\cdotp\varXi_{2}\cdotp|\log b_{1}|^{C}\lesssim b_{1}^{\frac{11}{3}},
\end{split}\]
\[\begin{split}
\int|H(\varepsilon^{3})|^{2}& \lesssim||\varepsilon||_{L^{\infty}}^{2}\sum_{i+j\leq2}\left(||\partial_{y}^{i}\varepsilon||_{L^{\infty}(y\le1)}^{2}\cdotp||\partial_{y}^{j}\varepsilon||_{L^{\infty}(y\le1)}^{2}+\int_{y\geq1}\frac{|\partial_{y}^{i}\varepsilon|^{2}}{y^{4-2i}}\cdotp||y^{j}\partial_{y}^{j}\varepsilon||_{L^{\infty}(y\ge1)}^{2}\right)\\
 & \lesssim b_{1}^{6}+|\log b_{1}|^{C}\varXi_{2}^{3}\lesssim b_{1}^{\frac{11}{3}}.
\end{split}\]
Summing up the above estimates, we get 
\begin{equation}
\int|H\mathscr{F}|^{2}\lesssim b_{1}^{\frac{11}{3}}.
\end{equation}
Thus 
\begin{equation}
\left|\int w_{2}H_{\lambda}({\lambda}^{-2}F_{\lambda})\right|=\frac{1}{\lambda^{4}}\left|\int\varepsilon_{2}H\mathscr{F}\right|\lesssim\frac{b_{1}^{\frac{7}{3}}}{\lambda^{4}},
\end{equation}
which completes the proof of (\ref{Lya2}). 
\end{proof}

\section{CLOSING THE BOOTSTRAP AND PROOF OF THEOREM 1.1}
\label{CLOSING THE BOOTSTRAP AND PROOF OF THEOREM 1.1}
\subsection{Closing the bootstrap}

Now we are able to close the bootstrap. As stated in section 3, this
consists of two steps, the first of which is the following: 
\begin{prop}
\label{prop:Improved-control}(Improved control) Assume $K>0$ has
been chosen large enough in Proposition \ref{prop:bootstrap}, then
for all $t\in[0,T_{exit})$ there hold
\begin{align}
\varXi_{1}(t)\leq&\sqrt{b_{1}(0)},\label{exit1'}\\
\varXi_{2}(t)\leq b_{1}^{\frac{4}{3}}|\log b_{1}|^{\frac{K}{2}},\quad\varXi_{4}(t)\leq& b_{1}^{4}|\log b_{1}|^{\frac{K}{2}},\quad\varXi_{6}(t)\leq\frac{K}{2}\frac{b_{1}^{6}}{|\log b_{1}|^{2}}.\label{exit2'}
\end{align}
\end{prop}

\begin{proof}
\emph{\uline{Step 1: Improved energy bound.}} (\ref{exit1'}) results
from the decrease of $E(u)$. Indeed, let $\tilde{\varepsilon}=\tilde{\alpha}+\varepsilon$,
then 
\begin{align}
E(u) & =\frac{1}{2}\int|\nabla u|^{2}-\frac{1}{4}\int|u|^{4}\nonumber\\ &=\frac{1}{2}\int|\nabla(Q+\tilde{\varepsilon})|^{2}-\frac{1}{4}\int(Q+\tilde{\varepsilon})^{4}\\ &=E(Q)+(H\tilde{\varepsilon},\tilde{\varepsilon})-\frac{1}{4}\int(4Q\tilde{\varepsilon}^{3}+\tilde{\varepsilon}^{4}).\nonumber 
\end{align}
We estimate these terms respectively. Using Cauchy-Schwarz inequality and the bound 
\[
|\partial_{y}^{i}\tilde{\alpha}|\lesssim\frac{b_{1}|\log b_{1}|^{C}}{1+y^{i}}\1_{y\le2B_{1}},
\]
the linear part is 
\[
(H\tilde{\varepsilon},\tilde{\varepsilon})=(H\varepsilon,\varepsilon)+(H\tilde{\alpha},\tilde{\alpha})+2(H\varepsilon,\tilde{\alpha})=(H\varepsilon,\varepsilon)+O(b_{1}^{\frac{1}{2}+\gamma}).
\]
for some $\gamma>0$ small enough. Moreover, the sub-coercivity (\ref{sub-coerH})
of $H$ and (\ref{exit3}) yield 
\[
(H\varepsilon,\varepsilon)\geq c\int|\partial_{y}\varepsilon|^{2}-\frac{1}{c}(\varepsilon,\psi)^{2}\gtrsim\varXi_{1}+O(\frac{b_{1}^{7}}{|\log b_{1}|^{2}}).
\]
For nonlinear terms, from (\ref{Hardy}), Sobolev embedding $\dot{H}^{1}(\mathbb{R}^{4})\hookrightarrow L^{4}(\mathbb{R}^{4})$
and assumption (\ref{exit1}) on $\varXi_{1}$ 
\[
\left|\int Q\tilde{\varepsilon}^{3}\right|\lesssim\int\frac{|\varepsilon|^{3}+|\tilde{\alpha}|^{3}}{1+y^{2}}\lesssim\int\frac{|\varepsilon|^{2}}{y^{2}}\cdotp||\varepsilon||_{L^{\infty}}+b_{1}^{2}|\log b_{1}|^{C}\lesssim b_{1}^{\frac{1}{2}+\gamma},
\]
\[
\int|\tilde{\varepsilon}|^{4}\lesssim\int|\varepsilon|^{4}+\int|\tilde{\alpha}|^{4}\lesssim\varXi_{1}^{2}+b_{1}^{2}|\log b_{1}|^{C}\lesssim b_{1}(0)
\]
With the above analysis and the smallness of $b_{1}$ ensured by (\ref{exit3}),
we conclude 
\begin{equation}
E(u)-E(Q)\gtrsim\varXi_{1}+O(b_{1}(0)^{\frac{1}{2}+\gamma}).
\end{equation}
On the other hand, by our construction of initial data and dissipation
of energy, we have
\begin{equation}
E(u)-E(Q)\lesssim b_{1}(0)|\log b_{1}(0)|^{C},
\end{equation}
and (\ref{exit1'}) follows by taking $b_{1}(0)$ small enough.

\emph{\uline{Step 2: Control of $\varXi_{6}$.}} Now we use proposition \ref{prop:Lya6} to obtain the improved
control for $\varXi_{6}$. To this end, we first derive an explicit
formula for $\lambda$, which is also useful for the later proof.

By modulation equation (\ref{Modeq1}) and explicit formula (\ref{decomb}), we obtain
\begin{equation}
\frac{\lambda_{s}}{\lambda}=-b_{1}+O(b_{1}^{3+\frac{1}{2}})=-\frac{2}{3s}+\frac{4}{9s\log s}+O\left(\frac{1}{s(\log s)^{\frac{5}{4}}}\right)\label{Lambda_s}
\end{equation}
or equivalently 
\begin{equation}
\left|\frac{d}{ds}\log \left[\frac{\lambda s^{\frac{2}{3}}}{(\log s)^{\frac{4}{9}}}\right]\right|\lesssim\frac{1}{s(\log s)^{\frac{5}{4}}}.
\end{equation}
Recall we assume $\lambda(0)=1$. Simple integration shows
\begin{equation}
\lambda=\frac{(\log s)^{\frac{4}{9}}}{s^{\frac{2}{3}}}\cdotp\frac{s_{0}^{\frac{2}{3}}}{(\log s_{0})^{\frac{4}{9}}}\left[1+O(\log s_0)^{-\frac{1}{4}}\right],\label{lambda}
\end{equation}
Now integrate (\ref{Lya6}) on $[0,t]$, we derive
\begin{equation}
\begin{split}
\varXi_{6}(t)\leq{} & \lambda^{10}(t)\left[\varXi_{6}(0)+\frac{b_{1}^{6}(0)}{|\log b_{1}(0)|^{2}}\right]+\frac{b_{1}^{6}(t)}{|\log b_{1}(t)|^{2}}\label{intLya6}\\
 & +C\left[1+\frac{K}{\sqrt{\log M}}+\sqrt{K}\right]\lambda^{10}(t)\int_{0}^{t}\frac{b_{1}}{\lambda^{12}}\frac{b_{1}^{6}}{|\log b_{1}|^{2}}(\tau)d\tau\nonumber 
\end{split}
\end{equation}
where $C>0$ is some constant independent of $M$ and $K$. 

Together with  (\ref{lambda}), (\ref{decomb}) and initial bound (\ref{initial2}),
we estimate the first term of (\ref{intLya6}) as 
\begin{equation}
\lambda^{10}(t)\left[\varXi_{6}(0)+\frac{b_{1}^{6}(0)}{|\log b_{1}(0)|^{2}}\right]\lesssim\lambda^{10}(t)\frac{b_{1}^{6}(0)}{|\log b_{1}(0)|^{2}}\lesssim\frac{b_{1}^{6}(t)}{|\log b_{1}(t)|^{2}}.\label{intLya6-1}
\end{equation}
For the integral term, we use (\ref{Modeq1}), (\ref{decomb}) to get
\[
\frac{(b_{1})_{s}}{b_{1}}=-\frac{3}{2}b_{1} +O\left(\frac{b_1}{|\log b_1|}\right).
\]
Combining with (\ref{Lambda_s}), this implies that  
\begin{equation}
\frac{d}{dt}\left(\frac{1}{\lambda^{10}}\frac{b_{1}^{6}}{|\log b_{1}|^{2}}\right)=\frac{1}{\lambda^{12}}\frac{b_{1}^{6}}{|\log b_{1}|^{2}}\left[-10\frac{\lambda_{s}}{\lambda}+(6+\frac{2}{|\log b_{1}|})\frac{(b_{1})_{s}}{b_{1}}\right]\gtrsim\frac{b_{1}}{\lambda^{12}}\frac{b_{1}^{6}}{|\log b_{1}|^{2}}.
\end{equation}
Therefore
\begin{equation}
\int_{0}^{t}\frac{b_{1}}{\lambda^{12}}\frac{b_{1}^{6}}{|\log b_{1}|^{2}}(\tau)d\tau\lesssim\frac{1}{\lambda^{10}(t)}\frac{b_{1}^{6}(t)}{|\log b_{1}(t)|^{2}}.\label{intLya6-2}
\end{equation}
Injecting (\ref{intLya6-1}), (\ref{intLya6-2}) into (\ref{intLya6}),
we conclude 
\begin{equation}
\varXi_{6}\leq C\left[1+\frac{K}{\sqrt{\log M}}+\sqrt{K}\right]\frac{b_{1}^{6}}{|\log b_{1}|^{2}},
\end{equation}
with $C>0$ independent of $M$ and $K$, and derive 
\begin{equation}
\varXi_{6}\leq\frac{K}{2}\frac{b_{1}^{6}}{|\log b_{1}|^{2}}
\end{equation}
by choosing $K$ large enough.

\emph{\uline{Step 3. Control of $\varXi_{4}$ and $\varXi_{2}$.}} Now we integrate (\ref{Lya4}) and (\ref{Lya2})
on $[0,t]$, and obtain 
\[
\varXi_{4}(t)\leq\lambda^{6}(t)\varXi_{4}(0)+C\lambda^{6}(t)\int_{0}^{t}\frac{b_{1}^{5}}{\lambda^{8}}(\tau)d\tau,
\]
\begin{equation}
\varXi_{2}(t)\leq\lambda^{2}(t)\varXi_{2}(0)+C\lambda^{2}(t)\int_{0}^{t}\frac{b_{1}^{\frac{5}{2}}}{\lambda^{4}}(\tau)d\tau.
\label{varXi2}
\end{equation}
We estimate the first term of each by (\ref{lambda}) and (\ref{initial2})
\[
\lambda^{6}(t)\varXi_{4}(0)\lesssim\lambda^{6}(t)b_{1}^{7}(0)\lesssim b_{1}^{4}(t)|\log b_{1}(t)|^{C},
\]
\[
\lambda^{2}(t)\varXi_{2}(0)\lesssim\lambda^{2}(t)b_{1}^{7}(0)\lesssim b_{1}^{\frac{4}{3}}|\log b_{1}|^{C}.
\]
Consequently, for some large constant $C>0$ (independent of $M$
and $K$), one has  
\[
\frac{d}{dt}\left(\frac{1}{\lambda^{6}}b_{1}^{4}|\log b_{1}|^{C}\right)=\frac{1}{\lambda^{8}}b_{1}^{4}|\log b_{1}|^{C}\left[-6\frac{\lambda_{s}}{\lambda}+(4-\frac{C}{|\log b_{1}|})\frac{(b_{1})_{s}}{b_{1}}\right]\gtrsim\frac{b_{1}^{5}}{\lambda^{8}}
\]
\[
\frac{d}{dt}\left(\frac{1}{\lambda^{2}}b_{1}^{\frac{4}{3}}|\log b_{1}|^{C}\right)=\frac{1}{\lambda^{4}}b_{1}^{\frac{4}{3}}|\log b_{1}|^{C}\left[-2\frac{\lambda_{s}}{\lambda}+(\frac{4}{3}-\frac{C}{|\log b_{1}|})\frac{(b_{1})_{s}}{b_{1}}\right]\gtrsim\frac{b_{1}^{\frac{7}{3}}}{\lambda^{4}}
\]
which bound the integral terms. To sum up, we obtain 
\[
\varXi_{4}\leq b_{1}^{4}|\log b_{1}|^{C},\quad \varXi_{2}\leq b_{1}^{\frac{4}{3}}|\log b_{1}|^{C}
\]
where $C$ is independent of $M$ and $K$. By taking $K$ large enough,
we conclude 
\[
\varXi_{4}\leq b_{1}^{4}|\log b_{1}|^{\frac{K}{2}},\quad \varXi_{2}\leq b_{1}^{\frac{4}{3}}|\log b_{1}|^{\frac{K}{2}}
\]
which finishes the whole proof of Proposition \ref{prop:Improved-control}. 
\end{proof}

Finally, we come to the heart of this paper, i.e. the control
of unstable models. The main point of our analysis is to separate
the two unstable directions $V_{2}$ and $\tau$ and to control both
of them at the same time based on a Brouwer argument. First, we shall analysis the dynamics of these two unstable
directions as a preparation. 
\begin{lem}
There holds,  for all $t\in[0,T_{exit})$, 
\begin{equation}
\left|\frac{d\tau}{ds}-\varsigma\tau\right|\lesssim\frac{b_{1}^{4}}{|\log b_{1}|}.\label{dynamictau}
\end{equation}
\end{lem}

\begin{proof}
By the dynamic equation (\ref{dynamicepsilon}) for $\varepsilon$,
and the fact $H\psi=-\varsigma\psi$, we get 
\begin{equation}
\tau_{s}-\varsigma\tau=(\partial_{s}\varepsilon,\psi)+(\varepsilon,H\psi)=(\partial_{s}\varepsilon+H\varepsilon,\psi)=(\mathscr{F},\psi)+\frac{\lambda_{s}}{\lambda}(\Lambda\varepsilon,\psi).
\end{equation}
In the proof of Proposition \ref{prop:Lya6} , in particular see (\ref{H3L}), (\ref{H3N1}), (\ref{H3N2})
and (\ref{claim2}), we have actually shown 
\[
\int|H^{3}\mathscr{F}|^{2}\leq C(M,K)\frac{b_{1}^{8}}{|\log b_{1}|^{2}},
\]
which gives the desired upper bound 
\[
|(\mathscr{F},\psi)|=\frac{1}{\varsigma^{3}}|(\mathscr{F},H^{3}\psi)|=\frac{1}{\varsigma^{3}}|(H^{3}\mathscr{F},\psi)|\lesssim||H^{3}\mathscr{F}||_{L^{2}}||\psi||_{L^{2}}\lesssim\frac{b_{1}^{4}}{|\log b_{1}|}.
\]
The second term is controlled using well localization of $\psi$. That is 
\[
\left|\frac{\lambda_{s}}{\lambda}(\Lambda\varepsilon,\psi)\right|\lesssim\left|\frac{\lambda_{s}}{\lambda}\right|\cdotp\int\frac{|\varepsilon|+|y\partial_{y}\varepsilon|}{1+y^{10}}\cdotp|(1+y^{10})\psi|\lesssim b_{1}\sqrt{\varXi_{6}}\lesssim\frac{b_{1}^{4}}{|\log b_{1}|}.
\]
Thus (\ref{dynamictau}) follows.
\end{proof}

For the dynamic of $V$, we use the improved $\tilde{b}_{2}$ in \eqref{defb2'}. Let
\begin{equation}
\tilde{b}_{1}:=b_{1},\quad \tilde{b}_{2}\text{ as in (\ref{defb2'})} ,\quad \tilde{b}_{3}:=0,
\end{equation}
with associated unstable models 
\begin{equation}
\tilde{b}_{k}=b_{k}^{e}+\frac{\tilde{U}_{k}}{s^{k}(\log s)^{\frac{5}{4}}},\quad k=1,2,\quad \tilde{U}_{3}=0,\quad\tilde{V}=P\tilde{U}=P\begin{bmatrix}\tilde{U}_{1}\\
\tilde{U}_{2}
\end{bmatrix}\label{4.19}
\end{equation}
and 
\begin{equation}
\tilde{\tau}(t):=\tau(t)\cdotp\frac{|\log b_{1}(t)|}{b_{1}(t)^{3+\frac{1}{2}}}\label{tau'}
\end{equation}
Moreover, we modify (\ref{exit3}) to
\begin{equation}
|\tilde{V}_{1}(t)|\leq1,\quad|\tilde{V}_{2}(t)|\leq1,\quad|\tilde{\tau}(t)|\leq1.\label{exit'V,tau}
\end{equation}
We define 
\[
\tilde{T}_{exit}:=\sup\{0\leq t_{1}\leq T(v_{0}):\ \forall\, t\in[0,t_{1}],\ \text{(\ref{exit1})-(\ref{exit3}) and (\ref{exit'V,tau}) hold}\}. 
\]
Recall the definition of $T_{exit}$ in Proposition \ref{prop:bootstrap}. We have the following lemma.
\begin{lem}\label{lemma：dynamicV'}
For all $t\in[0,T_{exit})$, there holds that
\begin{equation}
|\tilde{V}-V|\lesssim {s^{-\frac{1}{4}}}\label{dynamicV'-1}
\end{equation}
and 
\begin{equation}
|s(\tilde{V}_{k})_{s}-(D_A\tilde{V})_{k}|\lesssim{(\log s)^{-\frac{1}{4}}},\quad k=1,2.\label{dynamicV'-2}
\end{equation}
\end{lem}

\begin{proof}
(\ref{dynamicV'-1}) is immediately from (\ref{b2'-b2}) and $b_{1}\lesssim\frac{1}{s}$. To prove (\ref{dynamicV'-2}), we compute using the above definitions
\[
(\tilde{b}_{k})_{s}-(2k-1+c_{b_1})\tilde{b}_{1}\tilde{b}_{k}+\tilde{b}_{k+1}=\frac{1}{s^{k+1}(\log s)^{\frac{5}{4}}}\left[s(\tilde{U}_{k})_{s}-(A\tilde{U})_{k}+O(\frac{1}{\sqrt{\log s}}+\frac{|\tilde{U}|+|\tilde{U}|^{2}}{\log s})\right].
\]
Using (\ref{exit3}) and (\ref{dynamicV'-1}), we know $|\tilde{U}|\lesssim 1$.
then we use the modulation equations and $b_{1}\lesssim\frac{1}{s}$
to conclude 
\begin{align*}
|s(\tilde{U}_{k})_{s}-(A\tilde{U})_{k}| & \lesssim|(\tilde{b}_{k})_{s}-(2k-1+c_{b_1})\tilde{b}_{1}\tilde{b}_{k}+\tilde{b}_{k+1}|\cdotp s^{k+1}(\log s)^{\frac{5}{4}}+\frac{1}{\sqrt{\log s}}\\
 & \lesssim\frac{1}{\sqrt{|\log b_{1}|}}\frac{b_{1}^{3}}{|\log b_{1}|}\cdotp s^{k+1}(\log s)^{\frac{5}{4}}+\frac{1}{\sqrt{\log s}} \lesssim{(\log s)^{-\frac{1}{4}}}.
\end{align*}
This is equivalent to \eqref{dynamicV'-2}.
\end{proof}
Now we are ready to complete the second step towards closing the bootstrap. 
\begin{prop}
There exists some initial data $v_{0}$ such that $T_{exit}=T(v_{0})$. 
\end{prop}

\begin{proof}
We first construct proper $v_0$ so that (\ref{exit1})-(\ref{exit3}) and (\ref{exit'V,tau}) hold at $t=0$. Recall $v_{0}=\tilde{Q}_{b(0)}+\tau(0)\tilde{\psi}$ where $b_{k}(0)$
is determined by $U_{k}(0)$ and we have set $U_{1}(0)=0$.

Given a pair $(\tilde{V}_{2}(0),\tilde{\tau}(0))\in\mathbb{D}:=[-1,1]\times[-1,1]$,
we get the following from (\ref{tau'})
\begin{equation}
\tau(0)=\tilde{\tau}(0)\cdotp\frac{b_{1}(0)^{3+\frac{1}{2}}}{|\log b_{1}(0)|}.
\end{equation}
and from (\ref{defb2'}) and above definitions, we can determine $\tilde{V}_{1}(0)$
and $U_{2}(0)$ by solving 
\begin{equation}
\tilde{V}(0)=P\tilde{U}(0),\quad \text{with}
\begin{cases}
\tilde{U}_{1}(0)=U_{1}(0)=0,\\
\tilde{U}_{2}(0)=U_{2}(0)-\tau(0)\frac{(H^{2}\tilde{\psi}, \chi_{B_{\delta}}\Lambda Q)}{64\delta|
\log b_{1}(0)|}\cdotp s_{0}^{2}(\log s_{0})^{\frac{5}{4}}.
\end{cases}
\end{equation}
It is not hard to derive 
\begin{equation}
\begin{cases}
\tilde{V}_{1}(0)=-\frac{1}{3}\tilde{V}_{2}(0),\\
U_{2}(0)=\frac{1}{3}\tilde{V}_{2}(0)+\tau(0)\frac{(H^{2}\tilde{\psi}, \chi_{B_{\delta}}\Lambda Q)}{64\delta|\log b_{1}(0)|}\cdotp s_{0}^{2}(\log s_{0})^{\frac{5}{4}}.
\end{cases}
\end{equation}
With $U_{1}(0),U_{2}(0)$ and $\tau(0)$ at hand, $v_{0}$ is constructed
satisfying (\ref{exit1})-(\ref{exit2}) and (\ref{exit'V,tau}).

By contradiction, we assume $T_{exit}<T(v_{0})$ for all such $v_0$. Proposition \ref{prop:Improved-control}
and (\ref{dynamicV'-1}) imply $\tilde{T}_{exit}\le T_{exit}<T(v_{0})\le +\infty$.
Hence we obtain a map: 
\begin{align}
\begin{split}
\mathbb{D} & \to\partial\mathbb{D}\label{Brouwermap}\\
(\tilde{V}_{2}(0),\tilde{\tau}(0)) & \mapsto(\tilde{V}_{2}(\tilde{T}_{exit}),\tilde{\tau}(\tilde{T}_{exit})).
\end{split}
\end{align}
Moreover, the following strictly outgoing behavior is satisfied: For
$(\tilde{V}_{2}(\tilde{T}_{exit}),\tilde{\tau}(\tilde{T}_{exit}))\in\partial\mathbb{D}$,
there are two cases:

(i) If $|\tilde{V}_{2}(\tilde{T}_{exit})|=1$, using (\ref{dynamicV'-2}), we get
\[
|s(\tilde{V}_{2})_{s}-\frac{2}{3}\tilde{V}_{2}|\lesssim{(\log s)^{-\frac{1}{4}}},
\]
thus
\begin{equation}
\frac{d}{ds}\tilde{V}_{2}^{2}(\tilde{T}_{exit})>0.
\end{equation}

(ii) If $|\tilde{\tau}(\tilde{T}_{exit})|=1$ we resort to (\ref{dynamictau})
to compute 
\begin{equation}
\tilde{\tau}_{s}=\tilde{\tau}\left[\frac{\tau_{s}}{\tau}+\frac{(b_{1})_{s}}{b_{1}}(-3-\frac{1}{2}+\frac{1}{|\log b_{1}|})\right]=\varsigma\tilde{\tau}\left(1+O(b_1^{-\frac12})\right),
\end{equation}
which implies 
\begin{equation}
\frac{d}{ds}\tilde{\tau}^{2}(\tilde{T}_{exit})>0.
\end{equation}

Summing up the above analysis, we conclude that, the map defined by (\ref{Brouwermap})
is continuous. Combining with the Brouwer theorem, we get a contradiction
and thus finish the proof. 
\end{proof}

\subsection{Proof of Theorem \ref{thm:main}}
We choose initial data $v_{0}$ such that $T_{exit}=T(v_{0})\le+\infty $. As a result, our previous calculations are valid. Recall
\[
\frac{d}{ds}\log \frac{\lambda s^{\frac{2}{3}}}{(\log s)^{\frac{4}{9}}}=O\left(\frac{1}{s(\log s)^{\frac{5}{4}}}\right),
\]
and by integration on $[s,+\infty)$, we get
\begin{equation}
\lambda=c_{1}(v_{0})\frac{(\log s)^{\frac{4}{9}}}{s^{\frac{2}{3}}}\left[1+O({(\log s)^{-\frac{1}{4}}})\right]\label{lambda-1}
\end{equation}
for some constant $c_{1}(v_{0})>0$. Using $dt=\lambda^{2}ds$ and \eqref{Modeq1}, one has 
\[-\lambda\lambda_t=-\frac{\lambda_s}{\lambda}=\frac{2}{3s}\left[1+O\left(\frac{1}{\log s}\right)\right]=\frac{c_2(v_0)\lambda^{\frac{3}{2}}}{|\log \lambda|^{\frac{2}{3}}}\left[1+O((\log s)^{-\frac14})\right],\]
that is 
\[-\lambda^{-\frac{1}{2}}|\log \lambda|^{\frac{2}{3}}\lambda_t=c_2(v_0)(1+O(1)).\]
From this it is not hard to show $\lambda$ touches zero at some finite time $T=T(v_0)<\infty$.
Moreover
\begin{equation}
T(v_{0})-t=\int_{s}^{+\infty}\lambda^{2}d\sigma=c_{3}(v_{0})\frac{(\log s)^{\frac{8}{9}}}{s^{\frac{1}{3}}}\left[1+O((\log s)^{-\frac{1}{4}})\right].\label{T-t}
\end{equation}
Finally, we can conclude from  (\ref{lambda-1}) and (\ref{T-t}) that
\begin{equation}
\lambda(t)=c(v_{0})\frac{(T(v_{0})-t)^{2}}{|\log (T(v_{0})-t)|^{\frac{4}{3}}}(1+o_{t\to T(v_{0})}(1))
\end{equation}
This verifies (\ref{1.9}). 

To prove (\ref{1.8}), We adapt the strategy from \cite{merle2005profiles}. 
First, direct computation using (\ref{varXi2}) implies
\begin{equation}
\forall t\in[0,T),\quad \|\Delta\tilde{v}(t,x)\|_{L^{2}(\mathbb{R}^{4})}<C(v_{0})<+\infty
\label{Lapv}\end{equation}
where
\begin{equation}
\tilde{v}(t,x)=v(t,x)-\frac{1}{\lambda(t)}Q\left(\frac{x}{\lambda(t)}\right)
=(\tilde{\alpha}+\varepsilon)_{\lambda(t)}.
\end{equation}
Now standard parabolic theory ensures the regularity of $v(t,x)$ away from the origin. Hence
\begin{equation}
\forall R>0, \quad \tilde{v}(t,x)\to v^{*}\ \text{in}\ \dot{H}^{1}(|x|>R)
\end{equation}
for some $v^{*}\in\dot{H}^{1}(|x|>R)$.
Moreover, (\ref{Lapv}) ensures the energy of $\tilde{v}(t,x)$ does not concentrate at the origin, that is
 \begin{equation}
\forall t\in[0,T),\ \|\tilde{v}(t,x)\|_{\dot{H}^{1}(|x|<R)}\to 0\ \text{uniformly as}\  R\to 0
\end{equation}
Combined with the boundedness of energy, we have $v^*\in \dot{H^1}$ and
\begin{equation}
v(t,x)-\frac{1}{\lambda(t)}Q\left(\frac{x}{\lambda(t)}\right)\to v^{*} \quad\text{in $\dot{H}^{1}$} 
\end{equation}
as stated in (\ref{1.8}). This finishes the proof of Theorem \ref{thm:main}.

\section*{Acknowledgments}
The third author is very grateful to his advisors,
Chenjie Fan and Ping Zhang, for their constant encouragements and guidance. The third author is supported by National Key R\&D
Program of China under grant 2021YFA1000800.

\appendix

\section{SOME SOBOLEV LEMMAS}

Recall the definition of $\psi$ and $\Phi_{M}$ in (\ref{def:psi}) and (\ref{defPhi}).
The following lemma \ref{lem-subcoerH}-\ref{lem:weightedcoer} can be found in \cite{schweyer2012type}. Readers
can refer to the proof therein. 
\begin{lem}\label{lem-subcoerH}
(Sub-coercivity of $H$) There exists $M_{0}>0$ and $c>0$ such that
for $M>M_{0}$ and any $u\in\dot{H}_{rad}^{1}(\mathbb{R}^{4})$ with
$(u,\Phi_{M})=0$, we have 
\begin{equation}
(Hu,u)\ge c\int|\partial_{y}u|^{2}-\frac{1}{c}(u,\psi)^{2}.\label{sub-coerH}
\end{equation}
\end{lem}

\begin{lem}
(Hardy inequality) For all $u\in\dot{H}_{rad}^{1}(\mathbb{R}^{4})$,
we have 
\begin{equation}
\int\frac{|u|^{2}}{y^{2}}+\sup_{y\in\mathbb{R}^{4}}|yu|^{2}\lesssim\int|\partial_{y}u|^{2}.\label{Hardy}
\end{equation}
\end{lem}

\begin{lem}
For all $u\in\dot{H}_{rad}^{1}(\mathbb{R}^{4})\cap \dot{H}_{rad}^{2}(\mathbb{R}^{4})$ and $\gamma>0$,
we have 
\begin{equation}
\int|\partial_{y}^{2}u|^{2}+\int\frac{|\partial_{y}u|^{2}}{y^{2}}\lesssim\int|\Delta u|^{2},
\end{equation}
\begin{equation}
\int\frac{|u|^{2}}{y^{4}(1+|\log y|^{2})}\lesssim\int\frac{|\partial_{y}u|^{2}}{y^{2}}+\int_{1\leq y\leq2}|u|^{2},
\end{equation}
\begin{equation}
\int_{y\geq1}\frac{|u|^{2}}{y^{4+\gamma}(1+|\log y|^{2})}\lesssim\int_{y\geq1}\frac{|\partial_{y}u|^{2}}{y^{2+\gamma}(1+|\log y|^{2})}+\int_{1\leq y\leq2}|u|^{2}.
\end{equation}
\end{lem}

\begin{lem}\label{lem:weightedcoer}
(Weighted coercivity of $H$) Let $M\geq1$ be large enough, then
there exists $C(M)>0$ such that for any $u\in\dot{H}_{rad}^{1}(\mathbb{R}^{4})\cap\dot{H}_{rad}^{2}(\mathbb{R}^{4})$
with $(u,\Phi_{M})=0$, we have 
\begin{equation}
\int|Hu|^{2}\geq C(M)\left[\int\frac{|u|^{2}}{y^{4}(1+|\log y|^{2})}+\int\frac{|\partial_{y}u|^{2}}{y^{2}}+\int|\partial_{y}^{2}u|^{2}\right].
\end{equation}
\begin{align}
\begin{split}
\int_{y\leq1}&|Hu|^{2}+\int_{y\geq1}\frac{|Hu|^{2}}{y^{4+4k}(1+|\log y|^{2})}\\
\geq&\  C(M)\left[\ \int_{y\leq1}(|u|^{2}+|\partial_{y}u|^{2})+\int_{y\geq1}(\frac{|u|^{2}}{y^{8+4k}(1+|\log y|^{2})}+\frac{|\partial_{y}u|^{2}}{y^{6+4k}(1+|\log y|^{2})})\right].
\end{split}
\end{align}
\end{lem}

In the interior domain, we have the following estimates: 
\begin{lem}\label{lem-interior domain estimate}
(i) For any $u\in H_{rad}^{2k}(y\leq1)$, we have 
\begin{equation}
||u||_{H^{2k}(y\leq1)}\lesssim||u||_{L^{2}(y\leq1)}+||\Delta^{k}u||_{L^{2}(y\leq1)}.\label{interior1}
\end{equation}

(ii) For any $u\in H_{rad}^{2k+1}(y\leq1)$, we have 
\begin{equation}
||u||_{H^{2k+1}(y\leq1)}\lesssim||u||_{L^{2}(y\leq1)}+||\nabla\Delta^{k}u||_{L^{2}(y\leq1)}.\label{interior2}
\end{equation}
\end{lem}

\begin{proof}
(i) Using the standard PDE theory, we know that for $v\in H^{2k}(y\le1)$ with
$v|_{y=1}=0$ there holds 
\begin{equation}
||v||_{H^{2k}(y\leq1)}\lesssim||v||_{L^{2}(y\leq1)}+||\Delta v||_{H^{2k-2}(y\leq1)}.
\end{equation}
By induction, it is not hard to show that for any $v\in H^{2k}(y\leq1)$ with
\[
v=\Delta v=\Delta^{2}v=...=\Delta^{k-1}v=0\for\  y=1,
\]
we have 
\begin{equation}
||v||_{H^{2k}(y\leq1)}\lesssim||v||_{L^{2}(y\leq1)}+||\Delta^{k}v||_{L^{2}(y\leq1)},\label{interior1-1}
\end{equation}
Now we define $\varphi_{m}(y)=y^{2m}$ so that
\[
\Delta^{m}\varphi_{m}\neq 0\quad
\text{and}\quad \Delta^{n}\varphi_{m}=0,\  \forall\  n>m.
\]
Let $v:=u-\sum_{i=0}^{k-1}c_{i}\varphi_{i}$,
where $c_{i}$ is chosen such that 
\[
v=\Delta v=...=\Delta^{k-1}v=0\ \for\ y=1.
\]
which is equivalent to 
\begin{equation}
\begin{bmatrix}\varphi_{0} & \varphi_{1} & ... & \varphi_{k-1}\\
\Delta\varphi_{0} & \Delta\varphi_{1} & ... & \Delta\varphi_{k-1}\\
 & ... & ...\\
\Delta^{k-1}\varphi_{0} & \Delta^{k-1}\varphi_{1} & ... & \Delta^{k-1}\varphi_{k-1}
\end{bmatrix}\begin{bmatrix}c_{0}\\
c_{1}\\
\vdots\\
c_{k-1}
\end{bmatrix}=\begin{bmatrix}u(1)\\
\Delta u(1)\\
\vdots\\
\Delta^{k-1}u(1)
\end{bmatrix}.
\end{equation}
Note that the matrix above is upper triangular with nonzero diagonal
elements, so $c_{i}$ can be solved. Besides $\sum_{i=0}^{k-1}|c_{i}|\lesssim\sum_{i=0}^{k-1}|\Delta^{i}u(1)|$.

We apply (\ref{interior1-1}) to $v$ and conclude 
\begin{align}
||u||_{H^{2k}(y\leq1)} & \lesssim||v||_{H^{2k}(y\leq1)}+\sum_{i=0}^{k-1}|c_{i}|\\
 & \lesssim||v||_{L^{2}(y\leq1)}+||\Delta^{k}v||_{L^{2}(y\leq1)}+\sum_{i=0}^{k-1}|c_{i}|\nonumber \\
 & \lesssim||u||_{L^{2}(y\leq1)}+||\Delta^{k}u||_{L^{2}(y\leq1)}+\sum_{i=0}^{k-1}|\Delta^{i}u(1)|.\nonumber 
\end{align}
Elementary Sobolev interpolation shows 
\begin{equation}
\sum_{i=0}^{k-1}|\Delta^{i}u(1)|\lesssim||u||_{W^{2k-2,\infty}(y\leq1)}\lesssim\gamma||u||_{H^{2k}(y\leq1)}+C(\gamma)||u||_{L^{2}(y\leq1)},
\end{equation}
for any $\gamma>0$. Picking $\gamma$ small enough we obtain \eqref{interior1}.

(ii) The proof is completely similar based on the fact 
\[
||v||_{H^{2k+1}(y\le1)}\lesssim||v||_{L^{2}(y\le1)}+||\nabla\Delta^{k}v||_{L^{2}(y\le1)},\ \forall\, v\in H^{2k+1}(y\le1),\ v|_{y=1}=0,
\]
so we omit the details here. 
\end{proof}

\section{INTERPOLATION BOUNDS}

We recall that $\varepsilon$ satisfies the following orthogonality
conditions
\[
(H^{i}\varepsilon,\Phi_{M})=0,\quad \ i=0,1,2
\]
and upper bounds
\[
\varXi_{1}:=\int|\partial_{y}\varepsilon|^{2}\leq10\sqrt{b_{1}(0)},
\]
\[
\varXi_{2}\leq b_{1}^{\frac{4}{3}}|\log b_{1}|^{K},\quad\varXi_{4}\leq b_{1}^{4}|\log b_{1}|^{K}\quad\text{and}\quad\varXi_{6}\leq K\frac{b_{1}^{6}}{|\log b_{1}|^{2}},
\]
where $\varXi_{2k}=\int|H^{k}\varepsilon|^{2}$. We shall collect
relevant bounds for $\varepsilon$ needed in this paper. 
\begin{lem}
We have the following estimates with constants dependent of $M$,

(i) Weighted bounds for $H^{i}\varepsilon$: For $1\le k\le3,\ 0\leq i\leq k-1$, 
\begin{equation*}
\int_{y\leq1}(|H^{i}\varepsilon|^{2}+|\partial_{y}H^{i}\varepsilon|^{2})+\int_{y\geq1}\left(\frac{|H^{i}\varepsilon|^{2}}{y^{4k-4i}(1+|\log y|^{2})}+\frac{|\partial_{y}H^{i}\varepsilon|}{y^{4k-4i-2}(1+|\log y|^{2})}\right)\lesssim\varXi_{2k}.
\end{equation*}

(ii) Weighted bounds for $\partial_{y}^{i}\varepsilon$: 
\begin{equation}
\int\left(\frac{|\varepsilon|^{2}}{y^{4}(1+|\log y|^{2})}+\frac{|\partial_{y}\varepsilon|^{2}}{y^{2}}+|\partial_{y}^{2}\varepsilon|^{2}\right)\lesssim\varXi_{2},
\end{equation}
\begin{equation}
\int_{y\geq1}\frac{|\partial_{y}^{i}\varepsilon|^{2}}{y^{4k-2i}(1+|\log y|^{2})}\lesssim\varXi_{2k},\for\ k=2,3,\  0\leq i\leq2k,
\end{equation}
\begin{equation}
||\varepsilon||_{H^{6}(y\leq1)}^{2}\lesssim\varXi_{6}.
\end{equation}

(iii) Lossy bounds for $y\ge1$: 
\begin{equation}
\int_{y\geq1}\frac{1+|\log y|^{C}}{y^{12-2i}}|\partial_{y}^{i}\varepsilon|^{2}\lesssim|\log b_{1}|^{C_{1}(C)}\varXi_{6},\for\ 0\leq i\leq4,
\end{equation}
\begin{equation}
\int_{y\geq1}\frac{1+|\log y|^{C}}{y^{8-2i}}|\partial_{y}^{i}\varepsilon|^{2}\lesssim|\log b_{1}|^{C_{1}(C)}\varXi_{4},\for\ 0\leq i\leq2,
\end{equation}
\begin{equation}
\int_{y\geq1}\frac{1+|\log y|^{C}}{y^{4-2i}}|\partial_{y}^{i}\varepsilon|^{2}\lesssim|\log b_{1}|^{C_{1}(C)}\varXi_{2},\for\ i=0,1.
\end{equation}

(iv) Point-wise bounds for $y\ge1$: 
\begin{align}
\begin{split}
||\varepsilon(1+|\log y|^{C})||_{L^{\infty}(y\geq1)}^{2}+||y\partial_{y}\varepsilon(1+|\log y|^{C})||_{L^{\infty}(y\geq1)}^{2}\lesssim |\log b_{1}|^{C_{1}(C)}\varXi_{2},
\end{split}
\end{align}
\begin{align}
\left\|\frac{1+|\log y|^{C}}{y^{2}}\varepsilon\right\|_{L^{\infty}(y\geq1)}^{2}
& +\left\|\frac{1+|\log y|^{C}}{y}\partial_{y}\varepsilon\right\|_{L^{\infty}(y\geq1)}^{2}\nonumber\\
& +\left\|(1+|\log y|^{C})\partial_{y}^{2}\varepsilon\right\|_{L^{\infty}(y\geq1)}^{2}
\lesssim |\log b_{1}|^{C_{1}(C)}\varXi_4,
\end{align}
\begin{equation}
\left\|\frac{y\partial_{y}^{3}\varepsilon}{1+\log y}\right\|_{L^{\infty}(y\geq1)}^{2}\lesssim\varXi_{4}.
\end{equation}

(v) Point-wise bounds for $y\le1$: 
\begin{equation}
||\varepsilon||_{L^{\infty}(y\leq1)}^{2}+||\partial_{y}\varepsilon||_{L^{\infty}(y\leq1)}^{2}+||\partial_{y}^{2}\varepsilon||_{L^{\infty}(y\leq1)}^{2}+||\partial_{y}^{3}\varepsilon||_{L^{\infty}(y\leq1)}^{2}\lesssim\varXi_{6}.\label{pointwise2}
\end{equation}
\end{lem}

\begin{proof}
The proof is parallel to the Appendix B in \cite{schweyer2012type} with the help of
Appendix A in our article, so we omit the details here. 
\end{proof}

\small

\bibliographystyle{plainnat}
\bibliography{ref}

\end{document}